\documentclass[11pt,reqno]{amsart}
\usepackage[%
   pdftitle={{On general transport equations with
abstract boundary conditions. The case of divergence free force
field.}},
   pdfauthor={Luisa Arlotti, Jacek Banasiak and Bertrand Lods},
   pdfsubject={{Semigroups for general transport equations with abstract
boundary conditions}},
   pdfkeywords={Transport equation, Boundary
conditions, $C_0$-semigroups, Characteristic curves},
  colorlinks=true,
  linkcolor=red,
  citecolor=blue,
]{hyperref}
\usepackage[T1]{fontenc}
\usepackage{mathrsfs}
\usepackage{amsmath,amsthm,amssymb}
\usepackage{amscd,indentfirst,epsfig}
\def \X {\mathscr{E}}
\usepackage{latexsym,bm}
\usepackage{times}
\usepackage{enumerate}
\usepackage{mathrsfs}
\usepackage{stmaryrd}
\usepackage{amsopn}
\usepackage{mathtools}

\usepackage{amsmath}
\usepackage{amssymb}
\usepackage{amsfonts}
\usepackage{amsbsy}
\usepackage{amscd,indentfirst}
\usepackage{hyperref}
\usepackage{amsfonts,amsmath,latexsym,amssymb,verbatim,amsbsy}
\usepackage{amsthm}
\usepackage{colordvi}
\newtheorem{theo}{{Theorem}} [section]
\newtheorem{lemme}[theo]{{Lemma}}
\newtheorem{propo}[theo]{{Proposition}}
\newtheorem{cor}[theo]{{Corollary}}
\newtheorem{hyp}{{Assumption}}
\newtheorem{nb}[theo]{Remark}
\newtheorem{defi}[theo]{{Definition}}
\theoremstyle{definition}

\def \leq {\leqslant}
\def \geq {\geqslant}
\numberwithin{equation}{section}
\def\ind#1{\lower5pt\hbox{$\scriptstyle #1$}}

\def \O {\mathbf{\Omega}}
\newcommand{\Con}{\ensuremath{\mathscr{C}}}

\newcommand{\Drond}{\ensuremath{\mathscr{D}}}
\def \x {\mathbf{x}}
\def \z {\mathbf{z}}
\def\y {\mathbf{y}}
\def\lp {L^p_+}
\def\lm {L^p_-}
\def \ff {\mathscr{F}}

\def \d {\mathrm{d}}
\def \D {\Drond}
\def \G {\mathscr{G}}

\def \ml {M_{\lambda}}

\def \I {\mathcal{I}}

\def \l {\lambda}

\def \A {\mathcal{A}}
\def \uot {(U_0(t))_{t \geq 0}}
\def \uht {(U_H(t))_{t \geq 0}}

\def \ds {\displaystyle}
\def \T {\mathcal{T}}
\def \B {\mathsf{B}}
\def \t {\tau}
\newcommand{\verti}[1]{{\left\vert\kern-0.25ex\left\vert\kern-0.25ex\left\vert #1 
    \right\vert\kern-0.25ex\right\vert\kern-0.25ex\right\vert}}
 
\title[An $L^{p}$--approach to transport equations]{An $L^{p}$--approach to the well-posedness of transport equations associated to a regular field.}

\author{L. \textsc{Arlotti}    \& B. \textsc{ Lods}}

\address{\newline \noindent \textit{\textbf{Luisa Arlotti}} \newline \noindent Universit\`a di
Udine, via delle Scienze 208,\newline 33100 Udine, Italy.\newline
{\tt luisa.arlotti@uniud.it}}

\address{\newline \noindent \textit{\textbf{Bertrand Lods}}
\newline
\noindent  
\newline Universit\`{a} degli
Studi di Torino \& Collegio Carlo Alberto, Department ESOMAS, Corso Unione Sovietica, 218/bis, 10134 Torino, Italy.\newline
\noindent{\tt bertrand.lods@unito.it} }
\begin{document}
 \maketitle
\begin{abstract}
We investigate transport equations associated to a
Lipschitz field $\ff$ on some subspace of $\mathbb{R}^N$ endowed
with some general space measure $\mu$ in some $L^{p}$-spaces $1 < p <\infty$, extending the results obtained in two previous contributions \cite{mjm1,mjm2} in the $L^{1}$-context.  We notably prove the well-posedness
of  boundary-value transport problems with a large variety of boundary conditions. New explicit formula for the transport semigroup are in particular given.
\end{abstract}
\medskip

\tableofcontents
\section{Introduction}
\noindent

This paper  deals with the study in a $L^{p}$-setting ($p > 1$)  of general transport equation
\begin{subequations}\label{1}
\begin{equation}\label{1a}
\partial_t f(\x,t)+\mathscr{F}(\x)\cdot \nabla_\x
f(\x,t) =0 \qquad (\x \in \O , \:t
> 0),\end{equation} supplemented by the abstract boundary condition
\begin{equation}\label{1b}
f_{|\Gamma_-}(\y,t)=H(f_{|\Gamma_+})(\y,t), \qquad \qquad (\y \in
\Gamma_-, t >0),
\end{equation}
and the initial condition
\begin{equation}\label{1c}f(\x,0)=f_0(\x). \qquad \qquad (\x \in \O). \end{equation}\end{subequations}

The above problem was already examined by the authors in a $L^{1}$-setting in a series of papers  \cite{mjm1,mjm2}, and \cite{arlo11} ( \cite{mjm1,mjm2} in collaboration with J. Banasiak). Here our approach is generalized to the general $L^{p}$-space case for $1 < p <\infty$ which results in a complete study of the above set of equations \eqref{1} in general Lebesgue spaces.

Let us make precise the setting we are considering in the present paper, which somehow is the one considered earlier in \cite{mjm1,mjm2}. The set  $\O$ is a sufficiently smooth open subset of $\mathbb{R}^N$. We assume that $\mathbb{R}^N$ is endowed with a general
positive Radon measure $\mu$ and that  $\mathscr{F}$ is a
restriction to $\O$ of  a time independent globally Lipschitz vector
field $\mathscr{F} \::\:\mathbb{R}^N \to \mathbb{R}^N.$
 With this field we associate a
flow $(T_t)_{t \in \mathbb{R}} $ (with the notations of Section
\ref{sub:chara},
 $T_t=\Phi(\cdot, t)$) and, as in \cite{mjm1}, we
 assume the  measure $\mu$ to be invariant under the
flow $(T_t)_{t \in \mathbb{R}}$, i.e.
\begin{equation}\label{ass:h2} \mu(T_t A) = \mu(A) \text{ for any
measurable subset } A \subset \mathbb{R}^N \text{ and any }  t \in
\mathbb{R}.\end{equation}
The sets $\Gamma_{\pm}$ appearing in \eqref{1b} are suitable boundaries of the phase space and the
boundary operator $H$ is a linear, but not necessarily
\textit{bounded}, operator between trace spaces $L^{p}_\pm$
corresponding to the boundaries $\Gamma_\pm$ (see Section 2 for
details). 

We refer to the papers \cite{mjm1,mjm2} for the relevance of the above transport equation in mathematical physics (Vlasov-like equations) and the link of asssumption \eqref{ass:h2} with some generalized divergence free assumptions on $\ff$.  We also refer to the introduction of \cite{mjm2} and the references therein for an account of the relevant literature on the subject. 

Here we only recall that  transport problems of  type \eqref{1} with general fields  $\mathscr{F} $  have been studied in a $L^{p}$ context, in the special case in which the measure $\mu$ is the Lebesgue measure over $\mathbb{R}^N$, by  Bardos \cite{bardos} when the boundary conditions are the "no  re-entry" boundary conditions (i.e. $H = 0$), by Beals-Protopopescu  \cite{beals} when the boundary conditions are dissipative.

Furthermore an optimal trace theory  in $L^p$ spaces has been developed  by Cessenat \cite{ces1,ces2} for the so-called free transport equation i. e. when $\mu$ is the
Lebesgue measure  and 
$$\ff(x)=(v,0), \qquad \x=(r,v) \in \O$$
where $\O$ is a cylindrical domain of the type $\O= {D}\times
\mathbb{R}^3 \subset \mathbb{R}^6$ ( ${D }$ being a sufficiently
 smooth open subset of $\mathbb{R}^3$). 
 
 Recently a one-dimensional free transport equation with multiplicative boundary conditions has been examined  in $L^p$ by Boulanouar  \cite{bou1} but the criteria ensuring the well-posedness of the problem depend on $p$.
 
For more general fields and, most important for our present analysis, for more general and abstract measure, the mathematical treatment of \eqref{1} is much more delicate because of the intricate interplay between the
geometry of the domain and the flow as  their relation to the properties of the measure $\mu$.  Problems with  a general measure $\mu$ and general fields have been addressed only very recently in  \cite{mjm1,mjm2,arlo11} in a $L^{1}$-context.

The main aim of the present paper is to show that the theory and tools introduced in \cite{mjm1,mjm2} can be extended to $L^{p}$-spaces with $1 < p < \infty.$  The treatment of transport operators with $L^{p}$-spaces for $1 < p< \infty$ for abstract 
vector fields and abstract measure $\mu$ is new to our knowledge. As already mentioned in \cite{mjm2}, the motivation for studying such problems is to provide an unified treatment of first-order linear problems that should allow  to treat in the same formalism transport equations on an open subset of the euclidian space $\mathbb{R}^N$ (in such a case $\mu$ is a restriction of the Lebesgue measure over $\mathbb{R}^N$) and transport equations associated to flows on networks  where the measure $\mu$ is then supported on graphs   (see e.g. \cite{krammar} and the reference therein).

Besides showing the robustness of the theory developed in \cite{mjm1,mjm2},  the present contribution provides a thorough analysis of a large variety of boundary operators arising in first-order partial equations - including unbounded boundary operators, dissipative, conservative and multiplicative boundary operators. In order to obtain criteria ensuring the well-posedness of transport equations for conservative and multiplicative boundary conditions we use  a series representation of the solution to \eqref{1} introduced by the first author \cite{arlo11}  in the $L^{1}$ setting. The construction of such series representation is somehow reminiscent of the Dyson-Phillips representation of perturbed semigroups (see \cite{arloban}) and conforts us in our belief that -- for general transport equations -- boundary conditions have to be seen as ``boundary`` perturbation of the transport operator with no-reentry boundary conditions (see \cite{luisa1} where we adapted the substochastic theory of additive perturbations of $C_{0}$-semigroups to boundary perturbations). We refer to the seminal paper \cite{greiner} where boundary conditions  were already considered as perturbations of semigroup  generators.  The series approach to equation \eqref{1} allows us to get some semi-explicit expression of the solution to \eqref{1} like the following (see Corollary \ref{cor:semi}):
$$U_{H}(t)f(\x)=\begin{cases} U_{0}(t)f(\x)=f(\Phi(\x,-t)) \qquad &\text{ if } t < \t_{-}(\x)\\
\left[H\left(\B^{+}U_{H}(t-\t_{-}(\x))\right)f\right](\Phi(\x,-\t_{-}(\x))) \qquad &\text{ if } t \geq \t_{-}(\x).\end{cases}
$$
which holds for any $f\in \D(\T_{\mathrm{max}\,p})$ (see the subsequent section for notations).  Notice that such a representation was conjectured already by J. Voigt (see \cite[p. 103]{voigt}) for the free transport case. It has also been proved for a uni-dimensional population dynamics problem in $L^{1}$ with contractive boundary conditions (see \cite[Theorem 2. 3]{bou2}). To the best of our knowledge, for general fields and measures, multiplicative boundary conditions, in the $L^{p}$-setting the representation is new and has to been seen as one of the major contributions of the present paper.

The organization of the paper is as follows.  In Section \ref{sec:prelim} we
recall the relevant results of \cite{mjm1}: the definition of the
measures $\mu_\pm$ over $\Gamma_\pm$, the integration along the
characteristic curves associated to $\ff$. This allows us to give the precise definition of the
transport operator $\T_{\mathrm{max},\,p}$ in the $L^{p}$-context and gives the crucial link between $\T_{\mathrm{max},\,p}$ and the operator $\T_{\mathrm{max},\,1}$ which was throughly investigated in \cite{mjm1} (see Theorem \ref{th:lpl1}). In Section \ref{sec:bvp}, we apply the results of Section \ref{sec:prelim}  to prove well-posedness of the time--dependent transport problem with no reentry boundary conditions and we generalize the trace theory of Cessenat
\cite{ces1,ces2} to more general fields and measures. The
generalization is based on the construction of suitable trace spaces which are related to $L^{p}(\Gamma_{\pm},\d\mu_{\pm}).$ This allows to deal in Section \ref{sec:unb} with a
very large  class of boundary operators $H$, notably operators which are not necessarily bounded in the trace spaces $L^{p}(\Gamma_{\pm},\d\mu_{\pm}).$  This allows in particular to prove the well-posedness of \eqref{1} for general dissipative operators. Section \ref{sec:explicit}
 contains the construction of the series associated to \eqref{1} in the case of bounded $H$. This is done through some generalization of the Dyson-Phillips iterated. Notice that if the series is convergent then a $C_{0}$-semigroup solution to \eqref{1} can be defined. This happens for dissipative boundary operators and for some particular  conservative and multiplicative boundary operators. For a multiplicative boundary operator the sufficient condition ensuring  that a $C_{0}$-semigroup solution can be defined is the same as in the $L^1$-setting, and therefore independent from $p$.
\section{Preliminary results}\label{sec:prelim}.

We recall here the construction of characteristic curves, boundary measures $\mu_{\pm}$ and the maximal transport associated to \eqref{1} as established in \cite{mjm1}. Most of the results in this section can be seen as technical generalization to those of \cite{mjm1} and the proof of the main result of this section (Theorem \ref{representation}) is deferred to Appendix \ref{app:techn}. 
\subsection{Integration along characteristic curves}\label{sub:chara} The definition of the
transport operator (and the corresponding trace)  involved in
\eqref{1}, \cite{mjm1},  relies heavily on the characteristic
curves associated to the field $\ff$. Precisely,  define the flow
${\Phi}$ : $\mathbb{R}^N \times \mathbb{R} \to \mathbb{R}^N$, such
that, for $(\x,t) \in \mathbb{R}^N \times \mathbb{R}$, the mapping
$ t \in \mathbb {R} \longmapsto {\Phi}(\x,t) $ is the only
solution to the initial-value problem
\begin{equation}\label{chara}
\frac{\d \mathbf{{X}}}{\d t}(t)=\ff(\mathbf{{X}}(t)), \qquad \forall
t \in \mathbb{R}\:;\qquad \mathbf{{X}}(0)=\x \in \O.\end{equation}
Of course, solutions to \eqref{chara} do not necessarily belong to
$\O$ for all times, leading to the definition of stay times of the
characteristic curves in $\O$ as well as the related {\it incoming}
and {\it outgoing} parts of the boundary $\partial \O$:
\begin{defi}\label{incoming} For any $\x \in {\O}$, define
$\tau_{\pm}(\x)=\inf \{s > 0\,;{\Phi}(\x,\pm s) \notin {\O}\},$ with the
convention that $\inf \varnothing=\infty.$ Moreover, set
\begin{equation}\label{gammapm} \Gamma_{\pm}:=\left\{\y \in \partial
{\O}\,;\exists \x \in {\O},\, \t_{\pm}(\x) < \infty \text{ and }
\y={\Phi}(\x,\pm
\t_{\pm}(\x))\,\right\}.\end{equation}\end{defi} Notice that
the characteristic curves of the vector field $\ff$ are not
assumed to be of finite length and hence we introduce the sets
$${\O}_{\pm}=\{\x \in {\O}\,;\,\t_{\pm}(\x) < \infty\}, \qquad
{\O}_{\pm\infty}=\{\x \in {\O}\,;\,\t_{\pm}(\x) = \infty\},$$ and
$\Gamma_{\pm\infty}=\{\y \in \Gamma_{\pm}\,;\, \t_{\mp}(\y) = \infty\}.$
Then, one can prove (see \cite[Section 2]{mjm1}):
\begin{propo}\label{prointegra} There are unique positive Borel measures $\d\mu_{\pm}$
on $\Gamma_{\pm}$ such that, for any $h \in L^1({\O},\d \mu)$,
\begin{equation}\label{10.47}
\int_{{\O}_{\pm}}h(\x)\d\mu(\x)
=\int_{\Gamma_\pm}\d\mu_\pm(\y)\int_0^{\t_\mp(\y)}h\left({\Phi}(\y,\mp
s)\right)\d s,
\end{equation}
and
\begin{equation}\label{10.49}
\int_{{\O}_{\pm} \cap
{\O}_{\mp\infty}}h(\x)\d\mu(\x)=\int_{\Gamma_\pm\infty}\d\mu_\pm(\y)\int_0^{\infty}h\left({\Phi}(\y,\mp
s)\right)\d s.
\end{equation}
Moreover, for any $\psi \in L^1(\Gamma_-,\d\mu_-)$,
\begin{equation}\label{10.51}
\int_{\Gamma_-\setminus
\Gamma_{-\infty}}\psi(\y)\d\mu_-(\y)=\int_{\Gamma_+\setminus
\Gamma_{+\infty}}\psi({\Phi}(\z,-\t_-(\z)))\d\mu_+(\z).\end{equation}
\end{propo}

 \subsection{The maximal transport operator and trace
results}\label{sec:maxi}  The results of the previous section allow
us to define the (maximal) transport operator $\T_{\mathrm{max},\,p}$ as
the weak derivative along the characteristic curves.   To be
precise, let us define the space of \textit{test functions}
$\mathfrak{Y}$ as follows:

\begin{defi}[\textit{\textbf{Test--functions}}] Let $\mathfrak{Y}$ be the set of all measurable and bounded
functions $\psi :\O \to \mathbb{R}$ with compact support in $\O$ and such
that, for any $\x \in \O$, the mapping
$$s \in (-\t_-(\x),\t_+(\x)) \longmapsto \psi({{\Phi}}(\x,s))$$
is continuously differentiable with
\begin{equation}\label{defi:deriva}\x \in \O \longmapsto \dfrac{\d}{\d
s}\psi({{\Phi}}(\x,s))\bigg|_{s=0} \text{ measurable and
bounded}.\end{equation}
\end{defi}

In the next step we define the transport operator
$(\T_{\mathrm{max},p},\D(\T_{\mathrm{max},\,p}))$ in $L^{p}(\O,\d\mu)$, $p \geq 1.$
\begin{defi}[\textit{\textbf{Transport operator $\T_{\mathrm{max},\,p}$}}]
Given $p \geq 1,$ the domain of the maximal transport operator $\T_{\mathrm{max},\,p}$ is
the set $\D(\T_{\mathrm{max},\,p} )$ of all $f \in L^p(\O ,\d\mu)$ for
which there exists $g \in L^{p}(\O ,\d\mu)$ such that
$$\int_{\O }g(\x)\psi(\x)\d\mu(\x)=\int_{\O }f(\x)\dfrac{\d}{\d
  s}\psi({{\Phi}}(\x,s))\bigg|_{s=0} \d\mu(\x),\qquad \qquad \forall \psi \in
\mathfrak{Y}.$$ In this case, $g=:\T_{\mathrm{max},\,p} f.$
\end{defi}
\begin{nb}\label{rem:closed} It is easily seen that, with this definition, $(\T_{\mathrm{max},\,p},\D(\T_{\mathrm{max},\,p}))$ is a closed operator in $L^{p}(\O,\d\mu).$ Indeed, if $(f_{n})_{n} \subset \D(\T_{\mathrm{max},\,p})$ is such that 
$$\lim_{n}\|f_{n}-f\|_{L^{p}(\O,\d\mu)}=\lim_{n}\|\T_{\mathrm{max},\,p}f_{n}-g\|_{L^{p}(\O,\d\mu)}=0$$ for some $f,g \in L^{p}(\O,\d\mu)$, then for any test-function $\psi \in \mathfrak{Y}$, the identity 
$$\int_{\O}\T_{\mathrm{max},\,p}f_{n}(\x)\psi(\x)\d\mu(\x)=\int_{\O}f_{n}(\x)\dfrac{\d}{\d
  s}\psi({{\Phi}}(\x,s))\bigg|_{s=0} \d\mu(\x)$$
holds for any $n \in \mathbb{N}$. Taking the limit as $n\to\infty$, we obtain the identity
$$\int_{\O}g(\x)\psi(\x)\d\mu(\x)=\int_{\O}f(\x)\dfrac{\d}{\d
  s}\psi({{\Phi}}(\x,s))\bigg|_{s=0} \d\mu(\x)$$
  which proves that $f \in \D(\T_{\mathrm{max},\,p})$ with $g=\T_{\mathrm{max},\,p}f.$
\end{nb}

\subsection{{Fundamental representation formula: mild formulation}} Recall  that, if $f_1$ and $f_2$
are two functions defined over $\O$, we say that $f_2$ is a {\it
representative of} $f_1$ if $\mu\{\x \in \O\,;\,f_1(\x) \neq
f_2(\x)\}=0$, i.e. when $f_1(\x)=f_2(\x)$ for $\mu$-almost every $\x
\in \O $. The following fundamental result provides a characterization of  the
domain of $\D(\T_{\mathrm{max},\,p})$:
\begin{theo}\label{representation}
Let $f \in L^p(\O,\mu)$. The following are equivalent:
\begin{enumerate}\item There exists $g \in L^p(\O,\mu)$ and
a representative $f^\sharp$ of $f$ such that, for $\mu$-almost every
$\x \in \O$ and any $-\tau_-(\x) < t_1 \leq t_2 < \tau_+(\x)$:
\begin{equation}\label{integralTM-}f^{\sharp}({{\Phi}}(\x,t_1))-f^{\sharp}({{\Phi}}(\x,t_2))
=\int_{t_1}^{t_2}
g({{\Phi}}(\x,s))\d s.
\end{equation}
\item $f \in \D(\T_{\mathrm{max},p})$. In this case, $g=\T_{\mathrm{max},\,p} f$. 
\end{enumerate}
Moreover, if $f$ satisfies one of these equivalent condition, then
\begin{equation}\label{eq:limtrac}
\lim_{t \to 0+}f^\sharp({{\Phi}}(\y,t))\end{equation} exists  for
almost every $\y \in \Gamma_-$. Similarly, $\lim_{t \to
0+}f^\sharp({{\Phi}}(\y,-t))$ exists for almost every $\y \in
\Gamma_+$.
\end{theo}

 The proof of the theorem is made of several steps following the approach developed in \cite{mjm1} in the $L^{1}$-context. The extension to $L^{p}$-space with $1 < p < \infty$ is somehow a technical generalization and we refer to Appendix \ref{app:techn} for details of the proof. Notice that the existence of the limit 
\eqref{eq:limtrac} can be proven exactly as  in \cite[Proposition 3.16]{mjm1}

The above representation theorem allows to define the trace operators.
\begin{defi}
For any $f \in \D(\T_{\mathrm{max},\,p})$, define the \textit{traces}
$\B^{\pm}f$ by
\begin{equation*}
\B^+f(\y):=\lim_{t \to 0+}f^\sharp({{\Phi}}(\y,-t)) \qquad
\text{ and } \qquad \B^-f(\y):=\lim_{t \to
0+}f^\sharp({{\Phi}}(\y,t))
\end{equation*} for any $\y \in \Gamma_\pm$ for which the limits
exist, where   $f^\sharp$ is a suitable representative of $f$.
\end{defi}

Notice that, by virtue of \eqref{integralTM-}, it is clear that, for any $f \in \D(\T_{\mathrm{max},\,p})$, the traces $\B^{\pm}f$ on $\Gamma_{\pm}$ are well-defined and, for $\mu_\pm$-a.e. $\mathbf{{z}} \in \Gamma_\pm$,
\begin{equation}\label{traceform}
\B^\pm f(\mathbf{{z}})=f^\sharp({\Phi}(\mathbf{{z}},\mp
t))\mp\int_0^t[\T_{\mathrm{max},\,p}f]({\Phi}(\mathbf{{z}},\mp s))\d s, \qquad
\forall t \in (0,\tau_\mp(\mathbf{{z}})).\end{equation}

\subsection{Additional Properties}

An important general property of $T_{\mathrm{max},\,p}$ we shall need in the sequel is
given by the following proposition, which makes the link between $\T_{\mathrm{max},\,p}$ and the operator $\T_{\mathrm{max},\,1}$ studied in \cite{mjm1}.
\begin{theo}\label{th:lpl1}
Let $p \geq 1$ and $f \in \D(\T_{\mathrm{max},\,p})$. Then, $|f|^{p} \in \D(\T_{\mathrm{max},\,1})$ and
\begin{equation}\label{eq:tpt1}
\T_{\mathrm{max}\,,1}|f|^{p}=p\,\mathrm{sign}(f)\,|f|^{p-1}\,\T_{\mathrm{max},\,p}f\end{equation}
where $\mathrm{sign}(f)(\x)=1$ if $f(\x) > 0$ and $\mathrm{sign}(f)(\x)=-1$ if $f(\x) < 0$ $(\x \in \O).$
\end{theo} 
\begin{nb} Observe that, since both $f$ and $\T_{\mathrm{max}\,p}f$ belong to $L^{p}(\O,\d\mu)$ one sees that the right-hand-side of \eqref{eq:tpt1} indeed belongs to $L^{1}(\O,\d\mu).$\end{nb} 
\begin{proof}  
The proof follows the path of the version $p=1$ given in \cite[Proposition 2.2]{mjm2}. Let $f \in \D(\T_{\mathrm{max},\,p})$ and $\psi \in
\mathfrak{Y}$ be fixed. We shall denote by $f^\sharp$ the
representative of $f$ given by Theorem \ref{representation}. 
Using \eqref{10.47}, one has
\begin{equation*}\begin{split}
\int_{\O_-}|f(\x)|^{p}\dfrac{\d}{\d s}
\psi({\Phi}(\x,s))\big|_{s=0}\d\mu(\x)&=\int_{\Gamma_-}\d\mu_-(\y)\int_0^{\t_+(\y)}
|f({\Phi}(\y,t))|^{p} \dfrac{\d}{\d t}\psi({\Phi}(\y,t))\d t\\
&=\int_{\Gamma_-}\d\mu_-(\y)\int_0^{\t_+(\y)} |f^\sharp({\Phi}(\y,t))|^{p}
\dfrac{\d}{\d t}\psi({\Phi}(\y,t))\d t.
\end{split}
\end{equation*}
Let us fix $\y \in \Gamma_-$ and introduce $I_{\y}:=\{t \in (0,\t_+(\y))\,;\,f^\sharp({\Phi}(\y,t)) >0\}.$ As in \cite[Proposition 2.2]{mjm2}, there exists a sequence of mutually
disjoint intervals $(I_k(\y))_k=(s_k^-(\y), s_k^+(\y))_k\subset
(0,\tau_+(\y))$ such that
$$I_{\y}=\bigcup_{k \in \mathbb{N}}
\left(s_k^-(\y)\,,\,s_k^+(\y)\right).$$ We have
$$\int_{I_\y} |f^\sharp({\Phi}(\y,t))|^{p}
\dfrac{\d}{\d t}\psi({\Phi}(\y,t))\d t=\sum_{k} \int_{ s_k^-(\y)}^{
s_k^+(\y)} \left(f^\sharp({\Phi}(\y,t))\right)^{p} \dfrac{\d}{\d t}\psi({\Phi}(\y,t))\d
t.$$ Let us  distinguish several cases.
If $s_k^-(\y)\neq 0$ and $s_k^+(\y)\neq \tau_+(\y)$ then, from the
continuity of $t \mapsto f^\sharp({\Phi}(\y,t))$, we see that $
f^\sharp\left({\Phi}(\y,s_k^-(\y))\right)=f^\sharp\left({\Phi}(\y,s_k^+(\y))\right)=0$.
 Using \eqref{integralTM-} on the interval $(s_k^-(\y),s_k^+(\y))$, a simple integration by
parts leads to
\begin{equation}\label{sk-sk+}\int_{ s_k^-(\y)}^{ s_k^+(\y)} \left(f^\sharp({\Phi}(\y,t))\right)^{p}
\dfrac{\d}{\d t}\psi({\Phi}(\y,t))\d
t=\int_{s_k^-(\y)}^{s_k^+(\y)}\psi({\Phi}(\y,t))F({\Phi}(\y,t))
\d t.\end{equation}
where 
$$F:=p\,\mathrm{sign}(f^{\sharp})\,|f^{\sharp}|^{p-1}\,\T_{\mathrm{max},\,p}f.$$ As already observed, $F \in L^{1}(\O,\d\mu).$
 Next, we consider the case when $s_k^-(\y) =0$ or
$s_k^+(\y) =\tau_+(\y)<\infty$  for some $k$. Using the fact that $\psi$ is of compact support in $\O$ while    ${\Phi}(\y,s_k^+(\y))\in
\partial\O$, one proves again \eqref{sk-sk+} integrating by parts.
The last case to consider is  $s_k^+(\y)=\tau_+(\y) =\infty$ for
some $k$.  We shall use  \cite[Lemma 3.3]{mjm1} according to
which for $\mu_-$ almost every $\y\in \Gamma_-$ there is a sequence
$(t_n)_n$ such that $t_n\to \infty$ and $\psi({\Phi}(t_n,\y)) =0.$
Thus, focusing our attention on such $\y$s,  as in the proof of
\cite[Theorem 3.6]{mjm1}, integration by parts gives
$$\int_{s_k^-(\y)}^{t_n}
\left(f^\sharp({\Phi}(\y,t))\right)^{p} \dfrac{\d}{\d t}\psi({\Phi}(\y,t))\d
t=\int_{s_k^-(\y)}^{t_n}\psi({\Phi}(\y,t))F({\Phi}(\y,t))\d
t$$ for any $n$ and, by integrability of both sides, we prove Formula \eqref{sk-sk+} for $s_k^+(\y)=\tau_+(\y) =\infty$. In other words, \eqref{sk-sk+} is true for any $k \in \mathbb{N}$ and, summing up
over $\mathbb{N}$, we finally get
$$\int_{I_\y} |f^\sharp({\Phi}(\y,t))|^{p} \dfrac{\d}{\d
t}\psi({\Phi}(\y,t))\d
t=\int_{I_\y}\psi({\Phi}(\y,t))F({\Phi}(\y,t))
\d t.$$ Arguing in the same way on $J_\y=\{t \in(0,\t_+(\y))\,;\,f^\sharp({\Phi}(\y,t)) <0\}$ we get
$$\int_{J_\y} |f^\sharp({\Phi}(\y,t))|^{p}\dfrac{\d}{\d
t}\psi({\Phi}(\y,t))\d
t=-\int_{J_\y}\psi({\Phi}(\y,t))F({\Phi}(\y,t))
\d t$$ where, obviously, $|f^\sharp({\Phi}(\y,t))|=-
f^\sharp({\Phi}(\y,t)) $ for any $t \in J_\y$. Now, integration over
$\Gamma_-$ leads to
$$\int_{\O_-}\left|f(\x)\right|^{p}\dfrac{\d}{\d s}
\psi({\Phi}(\x,s))\big|_{s=0}\d\mu(\x)=\int_{\O_-}F(\x)\psi(\x)\d\mu(\x).$$ Using now
parametrization over $\Gamma_+$, we prove in the same way  that
$$\int_{\O_+ \cap \O_{-\infty}}|f(\x)|^{p}\dfrac{\d}{\d s}
\psi({\Phi}(\x,s))\big|_{s=0}\d\mu(\x)=\int_{\O_+ \cap
\O_{-\infty}}F(\x)\psi(\x)\d\mu(\x).$$ In the same way, following the proof of \cite[Proposition 2.2]{mjm2}, one gets that
$$\int_{\O_{+\infty} \cap
\O_{-\infty}}|f(\x)|^{p}\frac{\d}{\d s}
\psi(\Phi(\x,s))\big|_{s=0}\d\mu(\x)=
\int_{\O_{-\infty} \cap \O_{+\infty}}
F(\x)\,\psi(\x)\d\mu(\x),$$
where we notice that assumption \eqref{ass:h2} is crucial at this stage. Therefore, one sees that
$$\int_{\O} |f(\x)|^{p}\frac{\d}{\d s}
\psi(\Phi(\x,s))\big|_{s=0}\d\mu(\x)=\int_{\O} F(\x)\psi(\x)\d\mu(\x) \qquad \forall \psi \in \mathfrak{Y}.$$ Since $F \in L^{1}(\O,\d\mu)$, this exactly means that $|f|^{p} \in \D(\T_{\mathrm{max},\,1})$ with $\T_{\mathrm{max},\,1}|f|^{p}=F$ and the proof is complete.
  \end{proof}

We can now generalize Green's formula:
\begin{propo}[\textit{\textbf{Green's formula}}]\label{propgreen}
Let $f \in \D(\T_{\mathrm{max},\,p})$ satisfies $\B^-f \in \lm.$ Then
$\B^+f \in \lp$ and
\begin{equation}\label{greenform}
\|\B^{-}f\|_{\lm}^{p} - \|\B^{+}f\|_{\lp}^{p}=p\int_{\O} \mathrm{sign}(f)\,|f|^{p-1}\,\T_{\mathrm{max},\,p}f\,\d\mu.\end{equation}
\end{propo}
\begin{proof} Let $f \in \D(\T_{\mathrm{max},\,p})$ with $\B^{-}f\in \lm,$ be given. Let $F=|f|^{p}$. One checks without difficulty that  $|\B^{\pm}f|^{p}=\B^{\pm}|f|^{p}$ while, from the previous result $F \in \D(\T_{\mathrm{max},\,1})$. Since $\B^{-}F\in L^{1}_{-},$ applying the $L^{1}$-version of Green's formula \cite[Proposition 4.4]{mjm1}, we get
$$\int_{\O}\T_{\mathrm{max},1}|f|^{p}\d\mu=\int_{\Gamma_{-}}\B^{-}|f|^{p}\d\mu_{-}-\int_{\Gamma_{+}}\B^{+}|f|^{p}\d\mu_{+}$$
which gives exactly the results thanks to Theorem \ref{th:lpl1}.\end{proof}

\section{Well-posedness for initial and boundary--value problems}\label{sec:bvp}

\subsection{Absorption semigroup}\label{sec:exis}
From now on, we fix $p > 1$ and we will denote $X=L^p(\O,\d\mu)$ endowed with its
natural norm $\|\cdot\|_{p}$. The conjugate exponent will always be denoted by $q$, i.e. $1/p+1/q=1$. Let $\T_{0,\,p}$ be the free streaming
operator with \textit{no re--entry boundary conditions}:
$$\T_{0,\,p}\psi=\T_{\mathrm{max},\,p}\psi, \qquad \text{ for any }
\psi \in \D(\T_{0,\,p}),$$ where the domain $\D(\T_{0,\,p})$ is defined by
$$\D(\T_{0,\,p})=\{\psi \in \D(\T_{\mathrm{max},\,p})\,;\,\B^-\psi=0\}.$$
We state the following generation result:
\begin{theo}\label{uot} The operator $(\T_{0,\,p},\D(\T_{0,\,p}))$ is the generator of a nonnegative $C_0$-semigroup
of contractions $\uot$ in $L^{p}(\O,\d\mu)$ given by
\begin{equation}\label{eq:Uot}
U_0(t)f(\x)=f({{\Phi}}(\x,-t))\chi_{\{t <
\tau_-(\x)\}}(\x), \qquad (\x \in \O,\:f \in X),\end{equation} where $\chi_A$
denotes the characteristic function of a set $A$.
\end{theo}
\begin{proof}  Let us first check that the family of
operators $\uot$ is a nonnegative contractive $C_0$-semigroup in
$X$. As in \cite[Theorem 4.1]{mjm1}, for any
$f \in X$ and any $t \geq 0$, the mapping $U_0(t)f\::\O \to
\mathbb{R}$ is measurable and the semigroup properties $U_0(0)f=f$
and $U_0(t)U_0(s)f=U_0(t+s)f$ $(t,s \geq 0)$ hold. Let us now show
that $\|U_0(t)f\|_{p} \leq \|f\|_{p}$. We have
\begin{equation*}\label{eq:uot0}
\|U_0(t)f\|_{p}^{p}=\int_{\O_+}|U_0(t)f|^{p}\d\mu+\int_{\O_-\cap
\O_{+\infty}}|U_0(t)f|^{p}\d\mu+\int_{\O_{-\infty}\cap\O_{+\infty}}|U_0(t)f|^{p}\d\mu.\end{equation*}
As in \cite[Theorem 4.1]{mjm1}, one checks that
\begin{equation*}\label{eq:uot1}
\int_{\O_-\cap
\O_{+\infty}}|U_0(t)f|^{p}\d\mu= \int_{\O_-\cap
\O_{+\infty}}|f|^{p}\d\mu, \quad 
\int_{\O_{-\infty}\cap\O_{+\infty}}|U_0(t)f|^{p}\d\mu = \int_{\O_{-\infty}\cap\O_{+\infty}}|f|^{p}\d\mu.
\end{equation*} 
Therefore
$$\|f\|_{p}^{p}-\|U_{0}(t)f\|_{p}^{p}=\int_{\O_{+}}|f|^{p}\d\mu-\int_{\O_{+}}|U_{0}(t)|^{p}\d\mu.$$
Now, using \eqref{10.47} together with the expression of $U_{0}(t)f$ in \eqref{eq:Uot}, we get
$$\int_{\O_{+}}|U_{0}(t)f|^{p}\d\mu=\int_{\Gamma_{+}}\d\mu_{+}(\z)\int_{t}^{\max(t,\t_{-}(\z))}|f(\Phi(\z,-s))|^{p}\d s$$
so that
\begin{equation}\label{eq:f-UO}
\|f\|_{p}^{p}-\|U_{0}(t)f\|_{p}^{p}=\int_{\Gamma_{+}}\d\mu_{+}(\z)\int_{0}^{t}|f(\Phi(\z,-s))|^{p}\chi_{\{s < \t_{-}(\z)\}}\d s.\end{equation}
This proves that $\|U_{0}(t)f\|_{p} \leq \|f\|_{p}$, i.e. $\uot$ is a contraction semigroup. The rest of the proof is as in \cite[Theorem 4.1]{mjm1} since it involves only ``pointwise'' estimate.\end{proof}

\subsection{Some useful operators.} Introduce here some linear operators which will turn useful in the study of boundary-value problem. We start with
$$C_{\lambda}:=(\lambda-\T_{0,\,p})^{-1}, \qquad \forall \lambda >0.$$
Since
$$C_{\lambda}f=\int_{0}^{\infty}\exp(-\lambda\,t) U_{0}(t)f\d t, \qquad \forall f \in X,\;\;\lambda >0$$
one sees that
\begin{equation*}
\begin{cases}
C_{\lambda} \::\:&X \longrightarrow \D(\T_{0,\,p}) \subset X\\
&f \longmapsto \left[C_{\lambda}f\right](\x)=\displaystyle
\int_0^{\t_-(\x)}f({\Phi}(\x,-s))\exp(-\lambda s)\d s,\:\:\:\x \in
{\O}.
\end{cases}
\end{equation*}
In particular, $\|C_{\lambda}f\|_{p} \leq \frac{1}{\lambda}\|f\|_{p}$ for all $\lambda >0$, $f \in X.$
Introduce then, for all $f \in X$,
$$G_{\lambda}f=\B^{+}C_{\lambda}f, \qquad \forall \lambda >0, f \in X.$$
According to Green's formula, $G_{\lambda}f \in \lp$ and one has
\begin{equation*}
\begin{cases}
G_{\lambda} \::\:&X \longrightarrow \lp\\
&f \longmapsto \left[G_{\lambda}f\right]({\z})=\displaystyle
\int_0^{\t_-({\z})}f({\Phi}({\z},-s))\exp(-\lambda s)\d s,\:\:\:{\z}
\in \Gamma_+\;;\end{cases}
\end{equation*}
One has then the following
\begin{lemme}\label{lem:normGl} For any $\lambda >0$ and any $f \in X$, one has
\begin{equation}\label{eq:normGl}\|G_{\lambda}f\|_{\lp}^{p}+\lambda\,p\|C_{\lambda}f\|_{p}^{p}=p\int_{\O}\mathrm{sign}(C_{\lambda}f)\,|C_{\lambda}f|^{p-1}\,f\d\mu.\end{equation}
In particular
$$\|G_{\lambda}f\|_{\lp}^{p}+\lambda\,p\|C_{\lambda}f\|_{p}^{p} \leq p\|C_{\lambda}f\|_{p}^{p-1}\|f\|_{p}.$$
\end{lemme}
\begin{proof} Given $f \in X$ and $\lambda >0$, let $g=C_{\lambda}f=(\lambda-\T_{0,\,p})^{-1}f$. One has $g \in \D(\T_{0,\,p})$, i.e. $\B^{-}g=0$.  Green's formula (Proposition \ref{propgreen}) gives 
$$\|G_{\lambda}f\|_{\lp}^{p}=\|\B^{+}g\|_{\lp}^{p}=-p\int_{\O}\mathrm{sign}(g)|g|^{p-1}\T_{\mathrm{max},\,p}g\d\mu$$
and, since $\T_{\mathrm{max},\,p}g=\T_{0,\,p}g=\lambda\,g-f$, we get
$$\|G_{\lambda}f\|_{\lp}^{p}=-\lambda\,p\int_{\O}|g|^{p}\d\mu+p\int_{\O}\mathrm{sign}(g)|g|^{p-1}\,f\d\mu$$
which gives \eqref{eq:normGl}  since $g=C_{\lambda}f.$ The second part of the result comes from H\"older's inequality since $\mathrm{sign}(g)|g|^{p-1} \in L^{q}(\O,\d\mu)$ with $1/q+1/p=1$.
\end{proof}

\subsection{Generalized Cessenat's theorems}\label{sub:cessen}

The theory and tools we have recalled in the previous section allow us to carry out a more detailed study of the trace
operators. First of all, we show that
 Cessenat's trace result \cite{ces1,ces2} can be generalized to our case:

\begin{theo}\label{theotrace} Define the following measures over $\Gamma_{\pm}$:
$$\d\xi_{\pm}(\y)=\min\left(\t_{\mp}(\y),1\right)\d\mu_{\pm}(\y),
\qquad \y \in \Gamma_{\pm},$$
and set
$$Y^{\pm}_{p}:=L^p(\Gamma_{\pm},\d\xi_{\pm})$$
with usual norm. Then, for any $f \in \D(\T_{\mathrm{max},\,p})$,
the trace $\B^{\pm}f$ belongs to $Y^{\pm}_{p}$ with
$$\|\B^{\pm}f\|_{Y^{\pm}_{p}}^{p} \leq
2^{p-1}\left(\|f\|_{p}^{p}+\|\T_{\mathrm{max},\,p}f\|_{p}^{p}\right), \qquad f \in
\D(\T_{\mathrm{max},\,p}).$$
\end{theo}
\begin{proof} The proof is an almost straightforward application of
the representation formula \eqref{traceform}. The proof is similar to the one given in \cite[Theorem 3.1]{mjm2} for $p=1$. Namely,
let $f \in \D(\T_{\mathrm{max},\,p})$ be fixed. It is clear from
\eqref{traceform} that the mapping $\y \in \Gamma_{-} \mapsto
\B^-f(\y)$ is measurable. Now, for $\mu_-$-almost every $\y \in
\Gamma_-$, one has
$$\left|\B^-f(\y)\right|^{p} \leq 2^{p-1}|f^\sharp({\Phi}(\y,s))|^{p}+2^{p-1}\left(\int_0^s
|\T_{\mathrm{max},\,p}f({\Phi}(\y,r))|\d r\right)^{p}, \qquad \forall 0 < s <
\t_+(\y).$$ Now, for any $0 < s <t < \min(1,\t_+(\y))$,  using first H\"older inequality we get
$$\left(\int_0^s
|\T_{\mathrm{max},\,p}f({\Phi}(\y,r))|\d r\right)^{p} \leq s^{\frac{p}{q}}\,\int_{0}^{s}|\T_{\mathrm{max},\,p}f({\Phi}(\y,r))|^{p}\d r \leq \int_{0}^{s}|\T_{\mathrm{max},\,p}f({\Phi}(\y,r))|^{p}\d r.$$ Integrating
the above inequality with respect to $s$ over $(0,t)$ leads to
$$t\left|\B^-f(\y)\right|^{p}=\int_0^{t}\left|\B^-f(\y)\right|^{p}\d s \leq
2^{p-1}\int_0^{t}|f^\sharp({\Phi}(\y,s))|^{p}\d s + 2^{p-1}
 \int_0^{t} |\T_{\mathrm{max},\,p}f({\Phi}(\y,s))|^{p}\d
s$$ since $t \leq 1$. We conclude exactly as in the proof of \cite[Theorem 3.1]{mjm2}.
 \end{proof}
 
\begin{nb}\label{nb:graph-norm} A simple consequence of the above continuity result is the following: if $(f_{n})_{n} \subset \D(\T_{\mathrm{max},\,p})$ is such that 
$$\lim_{n}\left(\|f_{n}-f\|_{p} + \|\T_{\mathrm{max},\,p}f_{n}-\T_{\mathrm{max},\,p}f\|_{p}\right)=0$$
then $(\B^{\pm}f_{n})_{n}$ converge to $\B^{\pm}f$ in $Y^{\pm}_{p}.$\end{nb}

Clearly, 
\begin{equation}\label{eq:embed}
L^{p}_{\pm}=L^{p}(\Gamma_{\pm},\d\mu_{\pm}) \xhookrightarrow{} Y_{p}^{\pm}\end{equation}
where the embedding is continuous (it is a contraction). Define then, for all $\lambda > 0$  and any $u \in Y_{p}^{-}$:
\begin{equation*}\begin{cases}
\left[M_{\lambda}u\right](\z)&=u({\Phi}(\z,-\t_-(\z)))\exp\left(-\lambda
\t_-(\z)\right)\chi_{\{\t_-(\z) < \infty\}},\:\:\:\z \in \Gamma_+\;,\\
\left[\Xi_{\lambda}u\right](\x)&=u({\Phi}(\x,-\t_-(\x)))\exp\left(-\lambda
\t_-(\x)\right)\chi_{\{\t_-(\x)<\infty\}},\:\:\:\x \in {\O}\;.\end{cases}
\end{equation*}
We also  introduce the following measures on $\Gamma_{\pm}$:
$$\d\tilde{\mu}_{\pm,p}(\y):=\left(\min(\t_{\mp}(\y),\,1)\right)^{1-p}\,\d\mu_{\pm}(\y), \qquad \y \in \Gamma_{\pm}$$
and set $\widetilde{\mathcal{Y}}_{\pm,\,p}=L^{p}(\Gamma_{\pm},\,\d\tilde{\mu}_{\pm,p})$ with the usual norm. Notice that, $\d\tilde{\mu}_{\pm\,p}$ is absolutely continuous with respect to $\d\mu_{\pm}$ and the embedding
\begin{equation}\label{eq:embed}
\widetilde{\mathcal{Y}}_{\pm,\,p} \xhookrightarrow{} L^{p}(\Gamma_{\pm},\d\mu_{\pm})=L^{p}_{\pm}\end{equation}
 is continuous since it is a contraction.
One has the following result
\begin{lemme}\label{lemmaML}
Let $\lambda >0$ be given. Then,
$$M_\lambda \in \mathscr{B}(Y_{p}^{-},Y_{p}^{+}) \qquad \text{ and } \quad \Xi_{\lambda} \in \mathscr{B}(Y_{p}^{-},X).$$
Moreover, given $u \in Y_{p}^{-}$ it holds:
\begin{enumerate}
\item $\Xi_{\lambda}u \in
\D(\T_{\mathrm{max,\,p}})$ with 
\begin{equation}\label{propxil1}\T_{\mathrm{max},\,p}\Xi_{\lambda}u=\lambda \Xi_\lambda u, \qquad
\B^-\Xi_\lambda u=u,\:\:\B^+\Xi_\lambda u = \ml u.
\end{equation}
\item $M_\lambda u \in \lp$ if and only if
$u \in \lm$. 
\item  $M_\lambda u \in \widetilde{\mathcal{Y}}_{+,\,p}$ if and only if $u \in \widetilde{\mathcal{Y}}_{-,\,p}$.
\end{enumerate}
\end{lemme} 

\begin{proof} Let $\lambda >0$ and $u \in Y_{p}^{-}$ be fixed. From the definition of $\Xi_{\lambda}$ one sees that
\begin{equation}\label{f2-u}
|[\Xi_{\lambda}u]({\Phi}(\y,t))|^{p}=|u(\y)|^{p}\exp(-\lambda\,p\,t), \qquad \forall \y \in \Gamma_{-}, \quad 0 < t < \t_{+}(\y),
\end{equation}
and, thanks to Proposition \ref{prointegra}:
\begin{equation*}\begin{split}
\int_{\O}|[\Xi_{\lambda}u](\x)|^{p}\d\mu(\x)&=\int_{\Gamma_{-}}\d\mu_{-}(\y)\int_{0}^{\t_{+}(\y)}|u(\y)|^{p}\exp(-\lambda\,p\,t)\d t\\
&=\frac{1}{\lambda\,p}\int_{\Gamma_{-}}\left(1-\exp(-\lambda\,p\,\t_{+}(\y))\right)|u(\y)|^{p}\d\mu_{-}(\y)\\
&\leq \max\left(1,\frac{1}{\lambda\,p}\right)\int_{\Gamma_{-}}|u(\y)|^{p}\d\xi_{-}(\y)\end{split}\end{equation*}
where, as in \cite[Theorem 3.2]{mjm2}, we used that $(1-\exp(-s)) \leq \min(1,s)$ for all $s \geq 0.$ This shows in particular that 
$$\|\Xi_{\lambda}u\|_{p}^{p} \leq \max\left(1,\frac{1}{\lambda\,p}\right)\|u\|_{Y_{p}^{-}}^{p}.$$
Moreover, arguing as in \cite[Lemma 3.1]{mjm2} one has 
\begin{equation}\begin{split}
\label{f-2uh}
\int_{\Gamma_{+}}|[M_{\lambda}u](\z)|^{p}\d\xi_{+}(\z)&=\int_{\Gamma_{+}\setminus\Gamma_{+\infty}}\exp(-\lambda\,p\t_{-}(\z))|u({\Phi}(\x,-\t_{-}(\z))|^{p}\d\xi_{+}(\z)\\
&=\int_{\Gamma_{-}\setminus \Gamma_{-\infty}}\exp(-\lambda\,p\t_{+}(\y))|u(\y)|^{p}\d\xi_{-}(\y) \leq \|u\|_{Y_{p}^{-}}^{p}.\end{split}\end{equation}
This shows the first part of the Lemma. To prove that $\Xi_{\lambda}u \in \D(\T_{\mathrm{max},\,p})$ one
argues as in the proof of \cite[Theorem 4.2]{mjm1} to get that  $f_{2}=\Xi_{\lambda}u$ satisfies $\T_{\mathrm{max},\,p}f_{2}=\lambda\,f_{2}$. Moreover, it is clear from the definition of $\B^{+}$ that $\B^{+}\Xi_{\lambda}u=M_{\lambda}u.$ This shows point 1). To prove 2), we first notice that, for $u \in \lm$, as in \eqref{f-2uh}, one sees that
\begin{equation}\label{eq:mllm}
\int_{\Gamma_{+}}|[M_{\lambda}u](\z)|^{p}\d\xi_{+}(\z)=\int_{\Gamma_{-}\setminus\Gamma_{-\infty}}|u(\y)|^{p}\exp(-p\lambda\t_{+}(\y))\d\mu_{-}(\y) \leq \|u\|^{p}_{\lm}.\end{equation}
This, together with the embedding \eqref{eq:embed} shows that $M_{\lambda}u \in L^{p}_{+}$. Conversely, assume that $M_{\lambda}u \in \lp$ and define 
$$\Gamma_{-,1}=\{\y \in \Gamma_{-}\,;\,\t_{+}(\y) \leq 1\}, \qquad \Gamma_{-,2}=\Gamma_{-}\setminus \Gamma_{-,1}.$$
One has $\int_{\Gamma_{-,2}}|u(\y)|^{p}\d\xi_{-}(\y)=\int_{\Gamma_{-,2}}|u(\y)|^{p}\d\mu_{-}(\y) < \infty.$ Moreover, since $\lambda s+\exp(-\lambda s) \geq 1$ for any $s \geq 0$, one has
\begin{multline*}
\int_{\Gamma_{-,1}}|u(\y)|^{p}\d\mu_{-}(\y) \leq \int_{\Gamma_{-,1}}\left(\lambda\,p\t_{+}(\y)+\exp(-\lambda\,p\t_{+}(\y)\right)|u(\y)|^{p}\d\mu_{-}(\y) \\
\leq \lambda\,p \int_{\Gamma_{-,1}}|u(\y)|^{p}\d\xi_{-}(\y) \\
 + \int_{\Gamma_{-}\setminus \Gamma_{-\infty}}\exp(-\lambda\,p\t_{+}(\y))|u(\y)|^{p}\d\mu_{-}(\y)\\
 =\lambda\,p \int_{\Gamma_{-,1}}|u(\y)|^{p}\d\xi_{-}(\y) + \int_{\Gamma_{+}}|[M_{\lambda}u](\z)|^{p}\d\mu_{+}(\z)
\end{multline*}
according to \eqref{f-2uh}. This shows that $u \in \lm$ and proves the second point. 

It is clear now that, if $u \in \widetilde{Y}_{-,\,p}$ then $M_\lambda u \in \widetilde{Y}_{+,\,p}$. Conversely, assume that $M_\lambda u \in \widetilde{Y}_{+,\,p}$. To prove that $u \in \widetilde{Y}_{-,\,p}$, we only have to focus on the integral over $\Gamma_{+,1}$ since the measure $\d\tilde{\mu}_{+,p}$ and $\d\xi_{+}$ coincide on $\Gamma_{+,2}.$ Then, by assumption, it holds $I_{1} < \infty$ with
$$I_1:=\int_{\Gamma_{+,1}}|u(\Phi(\z,-\t_-(\z))|^p\exp\left(-p\lambda\,\t_-(\z)\right)\t_-(\z)^{1-p}\d\mu_+(\z).$$
Notice that $I_{1}$ can be written as
$$I_1=\int_{\Gamma_{-,1}}|u(\y)|^p\exp\left(-p\lambda\,\t_+(\y)\right)\t_+(\y)^{1-p}\d\mu_-(\y),$$
and, since $\exp\left(-p\lambda\,\t_+(\y)\right) \geq \exp(-\lambda\,p)$ for any $\y \in \Gamma_{-,1}$, we get 
$$\int_{\Gamma_{-,1}}|u(\y)|^p\d\tilde{\mu}_{-,\,p}(\y)=\int_{\Gamma_{-,1}}|u(\y)|^p \t_+(\y)^{1-p}\d\mu_-(\y) \leq \exp(\lambda\,p)I_1 < \infty.$$
As above, since $\d\tilde{\mu}_{-,\,p}$ coincide with $\d\xi_-$ on $\Gamma_{-,2}$, this shows that $u \in \widetilde{Y}_{-,\,p}$. \end{proof}

\subsection{Boundary-value problem} 

The above results allow  us to treat more general boundary-value
problems:
\begin{theo}\label{Theo10.43} Let $u \in Y_{p}^{-}$ and $g \in X$ be
given. Then the function
\begin{equation*}
f(\x)=\int_0^{\t_-(\x)}\exp(-\lambda t)\,g({{\Phi}}(\x,-t))\d
t +\chi_{\{\t_-(\x) < \infty\}} \exp(-\lambda \t_-(\x))
u({{\Phi}}(\x,-\t_-(\x)))\end{equation*} is a \textbf{
unique} solution $f \in \D(\T_{\mathrm{max},\,p})$ of the boundary value
problem:
\begin{equation}\label{BVP1}
\begin{cases}
(\lambda- \T_{\mathrm{max},\,p})f=g,\\
\B^-f=u,
\end{cases}
\end{equation}
where $\lambda > 0.$ Moreover, if $u \in \lm$ then
\begin{equation}\label{10.1156} \lambda\,p\,\|f\|_{p}^{p}+\|\B^{+}f\|_{\lp}^{p} \leq \|u\|_{\lm}^{p}+ p\|g\|_{p}\,\|f\|_{p}^{p-1}.
\end{equation}\end{theo}
\begin{proof} The fact that $f$ is the unique solution to \eqref{BVP1} is proven as in \cite[Theorem 4.2]{mjm1}. We recall here the main steps since we need the notations introduce therein. Write $f=f_1+f_2$ with $f_{1}=C_{\lambda}g$ and $f_{2}=\Xi_{\lambda}u.$ Since $f_1=(\lambda- \T_{0,\,p})^{-1}g,$
one has $f_1 \in \D(\T_{\mathrm{max},\,p})$ with $(\lambda-
\T_{\mathrm{max},\,p})f_1=g$ and $\B^-f_1=0$. Moreover, from  Proposition \ref{propgreen}, $\B^+f_1 \in
\lp$. On the other hand, according to Lemma \ref{lemmaML}, $f_2 \in
\D(\T_{\mathrm{max},\,p})$ with $(\lambda- \T_{\mathrm{max},\,p})f_2=0$ and
$\B^-f_2=u.$ We get then that  $f$ is a solution to \eqref{BVP1}. The uniqueness also follows the line of \cite[Theorem 4.2]{mjm1}.
Finally, it remains to prove \eqref{10.1156}. Recall that $f$ is a solution to \eqref{BVP1} and, applying Green formula \eqref{greenform} we obtain
$$\|u\|_{\lm}^{p}-\|\B^{+}f\|_{\lp}^{p}=p\int_{\O} \mathrm{sign}(f)\,|f|^{p-1}\,\T_{\mathrm{max},\,p}f\,\d\mu=p\int_{\O}\mathrm{sign}(f)\,|f|^{p-1}\left(\lambda\,f-g\right)\d\mu$$
i.e.
\begin{equation}\label{eq:greenBVP}\|u\|_{\lm}^{p}-\|\B^{+}f\|_{\lp}^{p}=p\,\lambda\|f\|_{p}^{p} - p\int_{\O} \mathrm{sign}(f)\,|f|^{p-1}\,g\,\d\mu\end{equation} 
which results easily in \eqref{10.1156}.\end{proof}
\begin{nb} Notice that, for $g=0$ and using the above notations, we have $f=f_{2}$ and \eqref{eq:greenBVP} reads
\begin{equation}\label{B+f-2}\lambda\,p \|f_2\|_{p}^{p} + \|\B^+f_2\|_{\lp}^{p}=\|u\|_{\lm}^{p}  < \infty.\end{equation} 
Conversely, assuming $u=0$, we get $f=f_{1}=C_{\lambda}g$ and $\B^{+}f=G_{\lambda}g$ and \eqref{eq:greenBVP} is nothing but Lemma \ref{lem:normGl}.\end{nb}

 \subsection{Additional properties of the traces} \label{sub:protrac}

 The generalization of  \cite[Proposition 2.3]{mjm2} to the case $p >1$ is the following:
\begin{propo}\label{fgt} Given $h \in \widetilde{\mathcal{Y}}_{+\,p}$, let
\begin{equation*}
f(\x)=\begin{cases} h({\Phi}(\x,\tau_+(\x)))
\dfrac{\tau_-(\x)e^{-\t_+(\x)}}{\t_-(\x)+\t_+(\x)}
\quad &\text{if } \: \tau_-(\x)+\tau_+(\x)< \infty, \\
 h({\Phi}(\x,\tau_+(\x))e^{-\t_+(\x)} \quad &\text{if }\:
\tau_-(\x)=\infty \text{ and } \tau_+(\x) < \infty,\\ 0 \quad
&\text{if }\: \tau_+(\x)=\infty.
\end{cases}\end{equation*}
Then, $f \in \D(\T_{\mathrm{max},\,p})$, $\B^-f=0$, and $\B^+f=h$. Moreover, 
$\|f\|_{p}  + \|\T_{0,\,p}f\|_{p}\leq 3\|h\|_{\widetilde{\mathcal{Y}}_{+,\,p}}$.\end{propo}
\begin{proof}  Let us first show that $f \in X$ with $\|f\|_{p}^{p} \leq \frac{1}{p}\|h\|_{\lp}^{p}$. We begin with noticing that
$$\int_{\O}|f(\x)|^{p}\d\mu(\x)=\int_{\O_+}|f(\x)|^{p}\d\mu(\x)=\int_{\O_+\cap
\O_-}|f(\x)|^{p}\d\mu(\x)+\int_{\O_+\cap
\O_{-\infty}}|f(\x)|^{p}\d\mu(\x),$$ since $f(\x)=0$ whenever
$\tau_+(\x)=\infty$. Now, according to the integration formula
\eqref{10.47},
\begin{equation*}\begin{split}
&\int_{\O_+\cap{\O}_-}|f(\x)|^{p}\d\mu(\x)=\int_{\O_+ \cap
\O_-}|h({\Phi}(\x,\tau_+(\x)))|^{p}
\dfrac{\tau_-(\x)^{p}e^{-p\t_+(\x)}}{\left(\t_-(\x)+\t_+(\x)\right)^{p}}\d\mu(x)\\
&=\int_{\Gamma_+ \setminus
\Gamma_{+\infty}}\d\mu_+({\z})\int_0^{\tau_-({\z})}\dfrac{|h({\z})|^{p}}{\tau_-({\z})^{p}}(\t_-({\z})-s)^{p}e^{-ps}\d
s.\end{split}\end{equation*}
Since, for any $\tau >0$ and any $s \in (0,\tau)$, $0 \leq 1 -\frac{s}{\t} \leq 1$, we have
$$\frac{1}{\tau^{p}}\int_{0}^{\tau}(\t-s)^{p}e^{-ps}\d s=\int_{0}^{\t}\left(1-\frac{s}{\t}\right)^{p}e^{-ps}\d s \leq \int_{0}^{\t} e^{-ps}\d s \leq \frac{1}{p}$$
we get that
$$\int_{\O_+\cap{\O}_-}|f(\x)|^{p}\d\mu(\x) \leq \frac{1}{p}\int_{\Gamma_+ \setminus
\Gamma_{+\infty}}|h(\z)|^{p}\d \mu_{+}(\z).$$
In the same way, according to Eq. \eqref{10.49},
\begin{equation*}\begin{split}
\int_{\O_+\cap \O_{-\infty}}|f(\x)|^{p}\d\mu(\x)&=\int_{\O_+\cap\O_{-\infty}}|h({\Phi}(\x,\tau_+(\x)))|^{p}
e^{-p\t_+(\x)}\d\mu(\x)\\
&=\int_{\Gamma_{+\infty}}\d\mu_+({\z})\int_0^{\infty}|h({\z})|^{p}e^{-ps}\d s=\frac{1}{p}\int_{\Gamma_{+\infty}}|h({\z})|^{p}\d\mu_+({\z}).
\end{split}\end{equation*}
Thus
$$\|f\|_{p}^{p} \leq \tfrac{1}{p}\|h\|_{\lp}^{p} \leq \tfrac{1}{p}\|h\|_{\widetilde{\mathcal{Y}}_{+,\,p}}^{p} \leq \|h\|_{\widetilde{\mathcal{Y}}_{+,\,p}}^{p}$$
and $f \in X$. Setting 
\begin{equation}\label{tmaxf}
 g(\x)=\begin{cases}
-h({\Phi}(\x,\tau_+(\x)))\,e^{-\t_+(\x)}
\dfrac{1+\tau_-(\x)}{\t_-(\x)+\t_+(\x)}\quad &\text{if } \: \tau_-(\x)+\tau_+(\x)< \infty, \\
 -h({\Phi}(\x,\tau_+(\x)))\,e^{-\t_+(\x)} \quad &\text{if }\:
\tau_-(\x)=\infty \text{ and } \tau_+(\x) < \infty,\\ 0 \quad
&\text{if }\: \tau_+(\x)=\infty,
\end{cases}
\end{equation}
it is easily seen that, if $g \in X$, then  $f\in \D(\T_{\mathrm{max},\,p})$ with $\T_{\mathrm{max},\,p}f=g.$ Let us then prove that $g \in X$. Clearly, as before,
$$\int_{\O_+\cap \O_{-\infty}}|g(\x)|^{p}\d\mu(\x)=\frac{1}{p}\int_{\Gamma_{+\infty}}|h(\z)|^{p}\d \mu_{+}(\z).$$
Moreover
\begin{equation*}\begin{split}
\int_{\O_+\cap \O_{-}}|g(\x)|^{p}\d\mu(\x)&=\int_{\O_+\cap\O_{-}}|h({\Phi}(\x,\tau_+(\x)))|^{p}
e^{-p\t_+(\x)}\frac{(1+\t_{-}(\x))^{p}}{(\t_{-}(\x)+\t_{+}(\x))^{p}}\d\mu(\x)\\
&=\int_{\Gamma_{+}\setminus \Gamma_{+\infty}}\d\mu_+({\z})\int_0^{\t_{-}(\z)}\frac{|h({\z})|^{p}}{\t_{-}(\z)^{p}}e^{-ps}(1+\t_-(\z)-s)^{p}\d s.\end{split}\end{equation*}
Now, it is easy to check that, for any $\t >0$,
$$\frac{1}{\t^{p}}\int_{0}^{\t}e^{-ps}(1+\t-s)^{p}\d s=\int_{0}^{\tau}\left(\frac{1+s}{\tau}\right)^{p}e^{-p(\tau-s)}\d s \leq 2^p\left[\min(\t,1)\right]^{1-p}$$
so that
$$\int_{\O_+\cap \O_{-}}|g(\x)|^{p}\d\mu(\x) \leq 2^p\int_{\Gamma_{+}\setminus\Gamma_{+\infty}}\left[\min(\t_{-}(\z),1)\right]^{1-p}\,|h(\z)|^{p}\d\mu_{+}(\z)\leq 2^{p}\|h\|^{p}_{\widetilde{\mathcal{Y}}_{+,\,p}}.$$
Consequently, we obtain
$$\int_{\O}|g(\x)|^{p} \d\mu(\x) \leq 2^p\|h\|^{p}_{\widetilde{\mathcal{Y}}_{+,\,p}} <\infty$$
which proves that $f \in \D(\T_{\mathrm{max},\,p})$. To compute the traces $\B^{+}f$ and $\B^{-}f$, one proceeds as in \cite[Proposition 2.3]{mjm2}. The inequality $\|f\|_{p}+\|\T_{\mathrm{max},\,p}f\|_{p}=\|f\|_{p}+\|g\|_{p} \leq 3\|h\|_{\widetilde{\mathcal{Y}}_{+,\,p}}$ is immediate.\end{proof}

\begin{nb} Notice that, for $p=1$, $\widetilde{\mathcal{Y}}_{\pm,1}=L^{1}(\Gamma_{+}\,,\d\mu_{\pm})$ and the above Propsition is nothing but \cite[Proposition 2.3]{mjm2}.\end{nb}

One also has the following
\begin{lemme}\label{glrange} For any $\lambda >0$ and $f \in X$, one has $G_{\lambda}f\in \widetilde{\mathcal{Y}}_{+,\,p}$ and
\begin{equation}\label{stime}
\|G_{\lambda}f\|_{\widetilde{\mathcal{Y}}_{+,\,p}} \leq \left(1+(\lambda\,q)^{-1/q}\right)\|f\|_{p}, \qquad \frac{1}{p}+\frac{1}{q}=1.\end{equation}
Moreover, for any $\lambda >0$ the mapping $G_{\lambda}\::\:X \to \widetilde{\mathcal{Y}}_{+,\,p}$ is surjective.
\end{lemme}
\begin{proof} Let $g \in \widetilde{\mathcal{Y}}_{+,\,p}$ and $\lambda >0.$ According  to Proposition \ref{fgt}, 
there is an $f \in \D(\T_{\mathrm{max},\,p})$, such that $\B^+f=g$ and $\B^{-}f=0$. In particular, since $f \in \D(\T_{0,\,p})$, there is $\psi \in X$ such that
$f=(\lambda-\T_{0,\,p})^{-1}\psi=C_{\lambda}\psi$. In this case, $g=\B^+f=G_\lambda \psi$. This proves that $G_{\lambda}\::\:X \to \widetilde{\mathcal{Y}}_{+,\,p}$ is surjective. Let us now prove \eqref{stime}: for $\lambda >0$ and $f \in X$ it holds
\begin{multline*}
\|G_{\lambda}f\|_{\widetilde{\mathcal{Y}}_{+,\,p}}^{p}=\int_{\Gamma_{+,1}}\left|\int_{0}^{\t_{-}(\z)}e^{-\lambda\,t}f(\Phi(z,-t))\d t\right|^{p}\t_{-}(\z)^{1-p}\d\mu_{+}(\z)\\
+ \int_{\Gamma_{+,2}}\left|\int_{0}^{\t_{-}(\z)}e^{-\lambda\,t}f(\Phi(z,-t))\d t\right|^{p}\d\mu_{+}(\z)\end{multline*}
where $\Gamma_{+,1}=\{\z \in \Gamma_{+}\,;\,\t_{-}(\z) \leq 1\}$ and $\Gamma_{+,2}=\{\z \in \Gamma_{+}\,;\,\t_{-}(\z) >1\}.$ For the first integral, we use Holder inequality to get, for any $\z \in \Gamma_{+,1}$:
$$\left|\int_{0}^{\t_{-}(\z)}e^{-\lambda\,t}f(\Phi(z,-t))\d t\right|^{p} \leq \left(\int_{0}^{\t_{-}(\z)}|f(\Phi(\z,-t)|^{p}e^{-\lambda\,pt}\d t\right)\t_{-}(\z)^{p/q}$$
i.e.
$$\left|\int_{0}^{\t_{-}(\z)}e^{-\lambda\,t}f(\Phi(z,-t))\d t\right|^{p} \leq \t_{-}(\z)^{p-1}\int_{0}^{\t_{-}(\z)}|f(\Phi(\z,-t)|^{p}\d t$$
and, using \eqref{10.47}, 
\begin{multline*}
\int_{\Gamma_{+,1}}\left|\int_{0}^{\t_{-}(\z)}e^{-\lambda\,t}f(\Phi(z,-t))\d t\right|^{p}\t_{-}(\z)^{1-p}\d\mu_{+}(\z)\\
 \leq \int_{\Gamma_{+,1}} \d\mu_{+}(\z)\int_{0}^{\t_{-}(\z)}|f(\Phi(\z,-t)|^{p}\d t \leq \|f\|_{p}^{p}.\end{multline*}
In the same way, we see that for all $\z \in \Gamma_{+,2}$, 
\begin{multline*}
\left|\int_{0}^{\t_{-}(\z)}e^{-\lambda\,t}f(\Phi(z,-t))\d t\right|^{p} \leq \left(\int_{0}^{\t_{-}(\z)}|f(\Phi(\z,-t)|^{p}\d t\right)\left(\int_{0}^{\t_{-}(\z)}e^{-\lambda\,qt}\d t\right)^{p/q}\\
\leq \left(\frac{1}{\lambda\,q}\right)^{p-1}\int_{0}^{\t_{-}(\z)}|f(\Phi(\z,-t)|^{p}\d t\end{multline*}
from which we deduce as above that
$$\int_{\Gamma_{+,2}}\left|\int_{0}^{\t_{-}(\z)}e^{-\lambda\,t}f(\Phi(z,-t))\d t\right|^{p}\d\mu_{+}(\z) \leq \left(\frac{1}{\lambda\,q}\right)^{p-1}\|f\|_{p}^{p}.$$
Combining both the estimates we obtain \eqref{stime}.
 \end{proof}

\begin{nb}\label{nb:glsurj} A clear consequence of the above Lemma is the following: if $\varphi \in \D(\T_{\mathrm{max},\,p})$ and $\B^{-}\varphi=0$ then $\B^{+}\varphi \in \widetilde{\mathcal{Y}}_{+,\,p}$.  \end{nb}
 The following result generalizes \cite[Theorem 2, p. 253]{dautray}:
\begin{propo}\label{prop1}
Let $\psi_{\pm} \in Y^{\pm}_{p}$ be given. There exists $\varphi \in
\D(\T_{\mathrm{max},\,p})$ such that $\B^{\pm}\varphi=\psi_{\pm}$ if and only if
$$\psi_+-\ml \psi_-  \in \widetilde{\mathcal{Y}}_{+,\,p} \qquad \text{ for some/all } \quad\lambda
>0.$$
\end{propo}
\begin{proof} Assume first there exists $\varphi \in \D(\T_{\mathrm{max},\,p})$
such that $\B^{\pm}\varphi=\psi_\pm$.  Set $g=\varphi-\Xi_\lambda
\psi_-$. Clearly, one can deduce from Eq. \eqref{propxil1} that $g
\in \D(\T_{\mathrm{max},\,p})$ with
$(\lambda-\T_{\mathrm{max},\,p})g=(\lambda-\T_{\mathrm{max},\,p})\varphi$ and
$\B^-g=0$. Moreover, $\B^+g=\psi_+-M_\lambda \psi_-$. Since $\B^-g=0$, one has $\B^+g \in \widetilde{\mathcal{Y}}_{+,\,p}$ (see Remark \ref{nb:glsurj}). Notice also that
$$\B^{+}g=G_{\lambda}(\lambda-\T_{\mathrm{max},\,p})\varphi$$
so that, from \eqref{stime}, 
\begin{multline}\label{stime1}
\|\B^{+}g\|_{\widetilde{\mathcal{Y}}_{+,\,p}} \leq (1+(\lambda\,q)^{-1/q})\|(\lambda-\T_{\mathrm{max},\,p})\varphi\|_{p}\\
\leq (1+(\lambda\,q)^{-1/q})\max(1,\lambda)\left(\|\varphi\|_{p}+\|\T_{\mathrm{max},\,p}\varphi\|_{p}\right).\end{multline}
Conversely, let $h=\psi_+-\ml \psi_- \in \widetilde{\mathcal{Y}}_{+,\,p}$. Then, thanks to
Proposition \ref{fgt}, one can find a function $f \in
\D(\T_{\mathrm{max},\,p})$ such that $\B^-f=0$ and $\B^+f=h$. Setting
$\varphi=f+\Xi_\lambda \psi_-$, one sees from \eqref{propxil1} that
$\varphi \in \D(\T_{\mathrm{max},\,p})$ with
$\B^-\varphi=\B^-f+\B^-\Xi_\lambda \psi_-=\psi_-$ and
$\B^+\varphi=h+\B^+\Xi_\lambda \psi_-=h+M_\lambda \psi_-=\psi_+$.\end{proof} 
A consequence of the above Proposition is the following:
\begin{cor} Given $\psi_{\pm} \in Y^{\pm}_p$, there exists $\varphi \in\D(\T_{\mathrm{max},\,p})$ such that $\B^{\pm}\varphi=\psi_{\pm}$ and $\B^{\mp}\varphi=0$ if and only if $\psi_{\pm} \in \widetilde{\mathcal{Y}}_{\pm,\,p}.$
\end{cor}
\begin{proof} The proof of the result is an obvious consequence of the above Proposition (see also Remark \ref{nb:glsurj}) and point (3) of Lemma \ref{lemmaML}.
\end{proof}
With this, we can prove the following:
\begin{propo}\label{propgreen1} One has $$\mathscr{W}:=\left\{f \in
\D(\T_{\mathrm{max},\,p})\,;\,\B^-f \in \lm\right\} =\left\{f \in
\D(\T_{\mathrm{max},\,p})\,;\,\B^+f \in \lp\right\}.$$ In particular, Green's formula \eqref{greenform} holds for any $f \in \mathscr{W}$.
\end{propo}
\begin{proof} We have already seen that, given $f \in \D(\T_{\mathrm{max},\,p})$ with $\B^-f \in
\lm$, it holds that $\B^{+}f\in \lp.$ Now, let $f \in \D(\T_{\mathrm{max},\,p})$ be such that
$\B^+f \in \lp.$ Set $u=\B^- f$. According to Proposition
\ref{prop1}, $\B^{+}f-M_\lambda u \in \widetilde{\mathcal{Y}}_{+,\,p} \subset \lp$. Therefore, $M_{\lambda}u \in \lp$ and since $u \in Y_{p}^{-}$, one deduces
from Lemma \ref{lemmaML} that $u \in \lm$.
\end{proof}

 Define now $\X$ as the space of elements
$(\psi_+,\psi_-) \in Y_{p}^{+} \times Y_{p}^{-}$ such that $\psi_+-\ml \psi_-
\in \widetilde{\mathcal{Y}}_{+,\,p}$ for some/all $\lambda
>0.$ We equip $\X$ with the norm
\begin{equation}\label{E}\|(\psi_+,\psi_-)\|_{\X}:=\left[\|\psi_+\|_{Y_p^{+}}^{p} + \|\psi_-\|_{Y_p^{-}}^{p}
+\|\psi_+-M_1 \psi_-\|_{\widetilde{\mathcal{Y}}_{+,\,p}}^{p}\right]^{1/p}\end{equation} that makes it a Banach
space. In the following, $\D(\T_{\mathrm{max},\,p})$ is endowed with the
graph norm: $\|f\|_{\D}:=\|f\|_{p}+\|\T_{\mathrm{max},\,p}f\|_{p}$, $f \in
\D(\T_{\mathrm{max},\,p})$.
\begin{cor}\label{coro1} The trace mapping $\mathbb{B} : \varphi \in \D(\T_{\mathrm{max},\,p})
\longmapsto (\B^+\varphi,\B^-\varphi) \in \X$ is continuous,
surjective with continuous inverse.
\end{cor}
\begin{proof} The fact that the trace mapping is surjective follows
from Proposition \ref{prop1}. Moreover, for any $\varphi \in
\D(\T_{\mathrm{max},\,p})$,  Theorem \ref{theotrace} yields
$\|\B^{\pm}\varphi\|_{Y_{p}^{\pm}} \leq 2^{1-\frac{1}{p}}\|\varphi\|_{\D}$. Moreover, according to \eqref{stime1}
$$\|\psi_+-\ml \psi_-\|_{\widetilde{\mathcal{Y}}_{+,\,p}}^{p}=\|\B^{+}g\|_{\widetilde{\mathcal{Y}}_{+,\,p}}^{p}\leq (1+(\lambda\,q)^{-1/q})^{p}\max(1,\lambda^{p})\|\varphi\|_{\D}^{p}$$
for any $\lambda >0.$ Choosing
$\lambda=1$, this proves that $\mathbb{B} $ is continuous. Conversely, suppose $(\psi_+,\psi_-)\in \X$. From the proof of Proposition \ref{prop1}
with $\lambda=1$ the inverse operator may be  defined by
$$
\mathbb{B}^{-1} \:\::\:\: (\psi_-,\psi_+) \mapsto \varphi=f+\Xi_1 \psi_-,
 $$
 where $f \in \D(\T_{\mathrm{max},\,p})$ satisfies $\B^-f =0$ and $\B^+f =
 h = \psi_+-M_1\psi_-$. Now, by Proposition \ref{fgt},
$$\|f\|_{\D} \leq 3 \|h \|_{\widetilde{\mathcal{Y}}_{+,\,p}} = 3\|\psi_+-M_1\psi_- \|_{\widetilde{\mathcal{Y}}_{+,\,p}}
$$ and one deduces easily the continuity of $\mathbb{B}^{-1}$.  \end{proof}

\section{Generation properties for unbounded boundary operators}\label{sec:unb}
Let us generalize some of the results of the previous section to  an unbounded operator $H$ from $Y_{p}^{+}$
to $Y_{p}^{-}$. We denote by $\D(H)$ its domain and $\G( H)$ its graph.
We assume in this section that \begin{equation} \G(H) \subset
\mathscr{E}.\label{scre}\end{equation} We define now $\T_{H,\,p}$ as
$\T_{H,\,p}f=\T_{\mathrm{max},\,p}f$ for any $f \in \D(\T_{H,\,p})$, where
$$\D(\T_{H,\,p})=\bigg\{f \in \D(\T_{\mathrm{max},\,p})\,;\,\mathbb{B}  f=(\B^+f,\B^-f) \in
\G(H)\bigg\}.$$  Notice that, thanks to Corollary \ref{coro1}, for
any $\psi \in \D(H)$, there exists $f \in\D(\T_{H,\,p})$ such that $\B^+
f=\psi.$

From now on, we equip $\D(H)$ with the norm:
\begin{equation}\label{normDH}\|\psi\|_{\D(H)}:=\|\psi\|_{Y_{p}^{+}} +
\|H\psi \|_{Y_{p}^{-}}+\|(I-\ml H)\psi \|_{\widetilde{\mathcal{Y}}_{+,\,p}}, \quad \psi  \in
\D(H),\end{equation} which is well-defined by (\ref{scre}). Then, one has the
following  whose proof is exactly the same as \cite[Lemma 4.1]{mjm2} (recall that $\mathscr{W}$ is defined in Proposition \ref{propgreen1}):
\begin{lemme}\label{lemdense} The set $\D(\T_{H,\,p}) \cap \mathscr{W}$ is dense in
$\D(\T_{H,\,p})$ endowed with the graph norm if and only if $\D(H) \cap \lp$ is dense in $\D(H)$. Moreover, for any
$\l
>0$, the following are equivalent:
\begin{enumerate}
\item $\left[I-\ml H\right]\D(H)=\widetilde{\mathcal{Y}}_{+,\,p}$;
\item $\mathrm{Ran}(\l-\T_{H,\,p})=X.$
\end{enumerate}
\end{lemme}
\begin{proof} The proof of the first point is  exactly the same as the one of \cite[Lemma 4.1]{mjm2} while the proof of the second one is exactly the one of \cite[Lemma 4.2]{mjm2}.\end{proof}

 We provide now necessary
and sufficient conditions on $H$ so that $\T_{H,\,p}$ generates a
$C_0$-semigroup of contractions in $X$. Our result generalizes
\cite[Theorem 3, p. 254]{dautray} in the context of $L^{p}$-spaces
but with general external field $\ff$ and Radon measure $\mu$.
\begin{theo}\label{TheoDautray} Let $H:\D(H) \subset Y_{p}^{+} \to Y_{p}^{-}$ be such that
\begin{enumerate}[(1)\:]
\item The graph $\G(H)$ of $H$ is a closed subspace of $\X$.
\item The range Ran$(I-\ml H)$ is a dense subset of  $\widetilde{\mathcal{Y}}_{+,\,p}$.
\item There is some positive constant $C >0$ such that
\begin{equation}\label{imlh}
\|(I-\ml H)\psi_+\|_{\widetilde{\mathcal{Y}}_{+,\,p}} \geq C \left(\|\psi_+\|_{Y_{p}^{+}} +
\|H\psi_+\|_{Y_{p}^{-}}\right), \qquad \forall \psi_+ \in \D(H).
\end{equation}
\item $\D(H) \cap \lp$ is dense in $\D(H)$ endowed with the norm
\eqref{normDH}.
\item The restriction of $H$ to $\lp$ is a contraction, i.e.
\begin{equation}\label{Hl1+contr} \|H\psi
\|_{\lm} \leq \|\psi \|_{\lp}, \qquad \forall \psi  \in \D(H) \cap
\lp.
\end{equation}
\end{enumerate}
Then, $\T_{H,\,p}$ generates a $C_0$-semigroup of contractions in $X$.
Conversely, if $\T_{H,\,p}$ generates a $C_0$-semigroup of contractions
and $\D(\T_{H,\,p}) \cap  \mathscr{W}$ is dense in $\D(\T_{H,\,p})$ endowed
with the graph norm, then $H$ satisfies assumptions
\textit{(1)--(5)}.
\end{theo}
\begin{proof} Assume  \textit{(1)--(5)} to hold.  Let $f \in \D(\T_{H,\,p}) \cap
\mathscr{W}$. Setting $g=(\l-\T_{H,\,p})f$, one sees that $f$ solves the
boundary-value problem \eqref{BVP1} with $u=H\B^{+}f$ and, from  \eqref{10.1156},
$$\l\,p \|f\|_{p}^{p} -\,p \|(\l-\T_{H,\,p})f\|_{p} \,\|f\|^{p-1}_{X}\leq
\|\B^-f\|_{\lm}-\|\B^+f\|_{\lp}=\|H(\B^+f)\|_{\lm}-\|\B^+f\|_{\lp}
\leq 0.$$ Thus, $\lambda\,\|f\|_{p}^{p} \leq \|(\l-\T_{H,\,p})f\|_{p}\,\|f\|_{p}^{p-1}$. This shows that $\lambda\|f\|_{p} \leq \|(\l-\T_{H,\,p})f\|_{p}$, i.e. $\T_{H,\,p}$ is dissipative over $\D(\T_{H,\,p}) \cap
\mathscr{W}$. From \textit{(4)} and Lemma  \ref{lemdense}, it is
clear that $\T_H$ is dissipative over $\D(\T_{H,\,p})$. Now, according to
\textit{(1)}, one sees that $\D(H)$ equipped with the norm
\eqref{normDH} is a Banach space. Moreover, for any $\l >0$, $I-\ml
H$ is continuous from $\D(H)$ into $\widetilde{\mathcal{Y}}_{+,\,p}$  and \textit{(2)-(3)} imply
that it is invertible with continuous inverse. In particular, since
Ran$(I-\ml H)=\widetilde{\mathcal{Y}}_{+,\,p}$, Lemma \ref{lemdense} implies that
Ran$(\l-\T_{H,\,p})=X$ so that the Lumer-Phillips Theorem \cite[p. 14]{pazy}
can be applied to state that $\T_{H,\,p}$ generates a  $C_0$-semigroup of contractions in
$X$.

Conversely, assume that $\T_{H,\,p}$ generates a $C_0$-semigroup of
contractions and $\D(\T_{H,\,p}) \cap  \mathscr{W}$ is dense in $\D(\T_{H,\,p})$ endowed with the graph norm.
According to the Lumer-Phillips Theorem, for any $f \in \D(\T_{H,\,p})$ and
any $g \in L^q(\O,\d\mu)$ with $\int_{\O}
f(\x)g(\x)\d\mu(\x)=\|f\|_{p}^{p} ,$ one has 
$$\int_{\O}g(\x)\T_{H,\,p} f(\x) \d\mu(\x)
\leq 0.$$ Then, for any $f \in \D(\T_{H,\,p}) \cap \mathscr{W}$,  choosing
$g=\mathrm{sign}f\,|f|^{p-1}$, Theorem \ref{th:lpl1} ensures that  
$g\,\T_{H,\,p}f=\frac{1}{p}\T_{H,\,1}(|f|^{p})$ so that
\begin{equation*}\begin{split}
0 \geq p\int_{\O}\T_{H,\,p}f(\x)g(\x)\d\mu(\x)&=\int_{\O}\T_{H,\,1}(|f|^{p})(\x)\d\mu(\x)\\
&=\int_{\Gamma_-}|\B^-f(\x)|^{p}\d\mu_-(\x)-\int_{\Gamma_+}|\B^+f(\x)|^{p}\d\mu_+(\x)\\
&=\|H\left(\B^+f\right)\|_{\lm}^{p}-\|\B^+f\|_{\lp}^{p}
\end{split}\end{equation*}
where we used Green's formula and the fact that
$\B^{\pm}|f|^{p}=\left|\B^{\pm}f\right|^{p}.$  This proves that
\eqref{Hl1+contr} holds for all $f \in \D(\T_{H,\,p}) \cap \mathscr{W}$. The rest of the proof is as in \cite[Theorem 4.1]{mjm2}. 
\end{proof}


\subsection{The transport operator associated to bounded boundary operators} \label{sec:bounded1}
Now we shall investigate in transport equations with  boundary operators  $H \in
\mathscr{B}(\lp,\lm)$. We denote for simplicity
$$\verti{H}=\|H\|_{\mathscr{B}(\lp,\lm)}.$$
More precisely, we introduce the associated transport operator $\T_{H,\,p}$ :
$$\T_{H,\,p}\psi=\T_{\mathrm{max},\,p}\psi, \qquad \text{ for any }
\psi \in \D(\T_{H,\,p}),$$ where the domain $\D(\T_{H,\,p})$ is defined by
$$\D(\T_{H,\,p})=\{\psi \in \D(\T_{\mathrm{max},\,p})\,;\,\B^+\psi \in \lp \quad \text{ and } \quad \B^{-}\psi=H\B^{+}\psi\}.$$

We have the following
\begin{propo} If $\verti{H} < 1$, then the operator $(\T_{H,\,p},\D(\T_{H,\,p}))$ is closed and densely defined in $X$.
\end{propo}
\begin{proof} 
Let $(f_{n})_{n}\subset \D(\T_{H,\,p})$ and $f,g \in X$ be such that $\lim_{n}\|f-f_{n}\|_{p}=\lim_{n}\|\T_{H,\,p}f_{n}-g\|_{p}=0.$ We have to prove that $f \in \D(\T_{H,\,p})$ with $\T_{H,\,p}f=g.$ Using the fact that $\T_{\mathrm{max},\,p}$ is closed (see Remark \ref{rem:closed}) and $\D(T_{H,\,p}) \subset \D(\T_{\mathrm{max},\,p})$ we get that $f \in \D(\T_{\mathrm{max},\,p})$ with $\T_{\mathrm{max},\,p}f=g.$ To prove the result, we ``only'' have to prove that $\B^{-} f \in \lm,$ $\B^{+}f \in \lp$ and $\B^{-}f=H\B^{+}f.$ First, according to Green's formula  one has, for any $n,m \geq 1$
$$\left|\|\B^{-}f_n-\B^{-}f_m\|_{\lm}^{p}-\|\B^{+}f_{n}-\B^{+}f_{m}\|_{\lp}^{p}\right| \leq p \|f_{n}-f_{m}\|_{p}^{p-1}\,\|\T_{H,\,p}f_{n}-\T_{H,\,p}f_{m}\|_{p}.$$
Since $\verti{H} < 1$, we have 
$$\left|\|\B^{-}f_n-\B^{-}f_m\|_{\lm}^{p}-\|\B^{+}f_{n}-\B^{+}f_{m}\|_{\lp}^{p}\right|\geq (1-\verti{H}^p)\|\B^+f_n-\B^+f_m\|_{\lp}^p.$$ In  particular, $(\B^{+}f_{n})_{n}$ is a Cauchy sequence in $\lp$ so it converges in $\lp.$ By definition of $\B^{\pm}f$ and since there is a subsequence (still denoted by $(f_{n})_{n}$) such that $\lim_{n}f_{n}(\x)=f(\x)$ for $\mu$-almost every $\x \in \O$, we get easily that the only possible limit is $\B^{+}f$, i.e $\lim_{n}\|\B^{+}f_{n}-\B^{+}f\|_{\lp}=0.$ Since $H$ is a bounded operator, we deduce that $\lim_{n}\|\B^{-}f_{n}-\B^{-}f\|_{\lm}=0$ and $H\B^{+}f=\B^{-}f$, i.e. $f \in \D(\T_{H,\,p}).$\medskip

Let us now prove that $\D(\T_{H,\,p})$ is dense in $X$. Notice that 
$$\D_{0}:=\{\psi \in \D(\T_{\mathrm{max},\,p})\,;\,\B^{-}\psi=\B^{+}\psi=0\} \subset \D(\T_{H,\,p}).$$
Now, since the set of continuously differentiable and compactly supported function $\mathcal{C}_{0}^1(\O)$ is dense in $X$ and $\mathcal{C}_{0}^1(\O) \subset \D_{0}$, we get the desired result.\end{proof}

In such a case, $\ml H \in
\mathscr{B}(\lp)$ for any $\l >0$. We begin with the following result:
\begin{propo}\label{prop:inverse}
Assume that $H \in \mathscr{B}(\lp,\lm)$. Let $\lambda >0$ be given such that $I-M_{\lambda}H \in \mathscr{B}(\lp)$ is invertible. Then, $(\lambda-\T_{H,\,p})$ is invertible and 
\begin{equation}\label{eqreso1}
(\lambda- \T_{H,\,p})^{-1}=C_{\lambda}+\Xi_{\lambda}H\left(I-\ml H
\right)^{-1}G_{\lambda}.\end{equation}
\end{propo}
\begin{proof} Let $\lambda >0$ be such that $I-M_{\lambda}H$ is invertible. Given $g \in X$, we wish to solve the resolvent equation
\begin{equation}\label{eq:re0}
(\lambda-\T_{H,\,p})f=g\end{equation}
for $f \in \D(\T_{H,\,p})$. This means that $f$ solves the boundary value problem $(\lambda-\T_{\mathrm{max},\,p})f=g$ with $\B^{-}f=H\B^{+}f$. If such a solution exists, it is given by 
\begin{equation}\label{eq:1}
f=C_{\lambda}g+\Xi_{\lambda}\B^{-}f\end{equation}
and therefore, taking the trace over $\Gamma_{+}$:
$$\B^{+}f=\B^{+}C_{\lambda}g+\B^{+}\Xi_{\lambda}\B^{-}f=G_{\lambda}g+M_{\lambda}\B^{-}f=G_{\lambda}g+M_{\lambda}H\B^{+}f.$$
Since $\B^{+}f \in \lp$ and $I-M_{\lambda}H$ is invertible, we get that $\B^{+}f$ is given by
\begin{equation}\label{eq:2}
\B^{+}f=(I-M_{\lambda}H)^{-1}G_{\lambda}g.\end{equation}
Then, inserting $\B^{-}f=H\B^{+}f$ into \eqref{eq:1}, we get that, if the resolvent equation \eqref{eq:re0} admits a solution, this solution is necessarily $f=C_{\lambda}g+\Xi_{\lambda}H(I-M_{\lambda}H)^{-1}G_{\lambda}g.$ Now, for any $g \in X$ and $\lambda >0$, we know that $f_{1}:=C_{\lambda}g \in \D(\T_{\mathrm{max},\,p})$ with $(\lambda-\T_{\mathrm{max},\,p})f_{1}=g.$ Since $G_{\lambda}g \in \lp$ one has $f_{2}:=\Xi_{\lambda}H(I-M_{\lambda}H)^{-1}G_{\lambda}g$ is well-defined and belongs to $X$. Moreover, according to Lemma \ref{lemmaML}, $f_{2} \in \D(\T_{\mathrm{max},\,p})$ with 
$\T_{\mathrm{max},\,p}f_{2}=\lambda f_{2}.$
This shows that, $f_{0}:=f_{1}+f_{2} \in \D(\T_{\mathrm{max},\,p})$ with $(\lambda-\T_{\mathrm{max},\,p})f_{0}=g$. Moreover, still using Lemma \ref{lemmaML}, $\B^{-}f_{1}=0$ while $\B^{-}f_{2}=H(I-M_{\lambda}H)^{-1}G_{\lambda}g \in \lm,$ i.e.
$$\B^{-}f_{0}=H(I-M_{\lambda}H)^{-1}G_{\lambda}g.$$
On the other-side, $\B^{+}f_{1}=G_{\lambda}g$ while 
$\B^{+}f_{2}=M_{\lambda}H(I-M_{\lambda}H)^{-1}G_{\lambda}g$
where we used again Lemma \ref{lemmaML}. Thus, $\B^{+}f_{0}=G_{\lambda}g+M_{\lambda}H(I-M_{\lambda}H)^{-1}G_{\lambda}g$ from which we deduce easily that
$$\B^{+}f_{0}=(I-M_{\lambda}H)^{-1}G_{\lambda}g.$$
Hence, $\B^{-}f_{0}=H\B^{+}f_{0}$ and $f_{0} \in \D(\T_{H,\,p})$. This proves that $f_{0}$ is indeed a solution to \eqref{eq:re0} and the proof is achieved.\end{proof}

If $H \in \mathscr{B}(\lp,\lm)$ is a contraction, one has the following:
\begin{theo}\label{Theo10.46} Let $H$ be strictly contractive, i.e.
$\verti{H} <1$. Then $\T_{H,\,p}$ generates a $C_0$-semigroup
of contractions $\uht$ in $X$ and the resolvent $(\lambda- \T_{H,\,p})^{-1}$ is
given by
\begin{equation}\label{eqreso}
(\lambda-
\T_{H,\,p})^{-1}=C_{\lambda}+\Xi_{\lambda}H\left(\sum_{n=0}^{\infty}(M_{\lambda}H)^nG_{\lambda}\right)
\qquad \text{ for any} \quad \lambda > 0,\end{equation} where the
series is convergent in
$\mathscr{B}(X).$
\end{theo}
\begin{proof} We provide two proofs of the result. A first direct one which relies on the Lumer-Phillips Theorem, the second one which comes from the result of Section \ref{sec:unb}.

\noindent$\bullet$ \textit{First proof: direct approach}. For any $\lambda >0$, one has $\|M_{\lambda}H\|_{\mathscr{B}(\lp)} \leq \verti{H} < 1$ so that $I-M_{\lambda}H \in \mathscr{B}(\lp)$ is invertible. According to Proposition \ref{prop:inverse}, we get that  $(\lambda-\T_{H,\,p})^{-1}$ is given by \eqref{eqreso} for any $\lambda >0.$ Let then $g \in X$ and $f=(\lambda-\T_{H,\,p})^{-1}g$. According to Green's formula (Prop. \ref{propgreen}), it holds
$$\|\B^{-}f\|_{\lm}^{p} - \|\B^{+}f\|_{\lp}^{p}=p\int_{\O} \mathrm{sign}(f)\,|f|^{p-1}\,\T_{\mathrm{max},\,p}f\,\d\mu,$$
which, using that $\T_{\mathrm{max},\,p}f=\lambda\,f-g$ and $\B^{-}f=H\B^{+}f$, reads
$$\|H\B^{+}f\|_{\lm}^{p}-\|\B^{+}f\|_{\lp}^{p}=\lambda\,p\,\|f\|_{p}^{p}-p\int_{\O} \mathrm{sign}(f)\,|f|^{p-1}\,g\d\mu.$$
Now, $\|H\B^{+}f\|_{\lm}^{p}-\|\B^{+}f\|_{\lp}^{p} \leq 0$ since $H$ is a contraction so that
$$\lambda\,p\|f\|_{p}^{p} \leq p\int_{\O} \mathrm{sign}(f)\,|f|^{p-1}\,g\d\mu.$$
A simple use of Holder's inequality gives then $\lambda\|f\|_{p}^{p} \leq \|f\|_{p}^{p-1}\,\|g\|_{p}$ and therefore
$$\|(\lambda-\T_{H,\,p})^{-1}\|_{\mathscr{B}(X)} \leq \frac{1}{\lambda} \qquad \forall \lambda >0.$$
Since moreover $\T_{H,\,p}$ is a densely defined and closed operator, a simple consequence of Lumer-Phillips Theorem asserts that $\T_{H,\,p}$ is the generator of a contraction $C_{0}$-semigroup in $X$.\\

\noindent$\bullet$ \textit{Second proof:}  The proof is a simple consequence of the above
Theorem \ref{TheoDautray} since, when $H$ is a strict contraction, one
has $\|M_{\lambda} H\|_{\mathscr{B}(\lp)} < 1$ for any $\lambda >0$ which, thanks to Hadamard's criterion, ensures that that  $(I-\ml H)$ is invertible
with inverse $(I-\ml H)^{-1}=\sum_{n=0}^\infty (\ml H)^n$ for any
$\l
>0$.   
\end{proof}

We will see in the next section an extension of the above result together with an explicit expression of the associated semigroup.

\section{Explicit transport semigroup for bounded boundary operators}\label{sec:explicit}

\subsection{Boundary Dyson-Phillips iterated operators}  Introduce as earlier the set
$$\D_{0}:=\{\psi \in \D(\T_{\mathrm{max},\,p})\,;\,\B^{-}\psi=\B^{+}\psi=0\} \subset \D(\T_{H,\,p}).$$

Recall now that, from Theorem \ref{uot}, $(\T_{0,\,p},\D(\T_{0,\,p}))$ generates a $C_{0}$-semigroup $(U_{0}(t))_{t\geq 0}$ in $X$ given by \eqref{eq:Uot}. Notice that, for any $f \in \D_{0}$, 
$$U_{0}(t)f \in \D(\T_{0,\,p}) \qquad \forall t \geq 0.$$
In particular, $\B^{\pm}U_{0}(t)f \in Y_{p}^{\pm}$. We set 
$$\I^{0}_{t}[f]=\int_{0}^{t}U_{0}(s)f\d s, \qquad \forall t \geq 0, \:f \in X.$$
Recall that, as a general property of $C_{0}$-semigroups (see for instance \cite[Theorem 2.4.]{pazy}):
$$\I^{0}_{t}[f] \in \D(\T_{0,\,p}) \qquad \text{ with } \quad \T_{0,\,p}\I^{0}_{t}[f]=U_{0}(t)f-f.$$
One has the following
\begin{propo}\label{prop:B+U01}
For any $f \in X$ and any $t >0$, the traces $\B^{\pm}\I^{0}_{t}[f] \in L^{p}_{\pm}$ and the mappings $t \geq 0 \mapsto \B^{\pm}\I^{0}_{t}[f] \in L^{p}_{\pm}$ are continuous. Moreover
\begin{equation}\label{eq:b+I}
\left\|\B^{+}\left(\I^{0}_{t+h}[f]-\I^{0}_{t}[f]\right)\right\|_{\lp}^{p} \leq h^{p-1}\,\left|\,\|U_{0}(t+h)f\|_{p}^{p}-\|U_{0}(t)f\|_{p}^{p}\,\right| \qquad \forall h >0.\end{equation}
\end{propo}
\begin{proof} For $f \in X$, since $\I^{0}_{t}[f] \in \D(\T_{0,\,p})$, one has $\B^{-}\I^{0}_{t}[f]=0.$ In particular, the trace $\B^{-}\I^{0}_{t}[f]$ belongs to $\lm$ with $t \geq 0 \mapsto \B^{-}\I^{0}_{t}[f] \in \lm$ continuous. According to Proposition \ref{prop:Ut}, $\B^{+}\I^{0}_{t}[f] \in \lp$ for all $t \geq 0$ and the mapping $t \geq 0 \mapsto \B^{+}\I^{0}_{t}[f] \in \lp$ is continuous.   
Given $0 \leq t < t+h$, one has
\begin{multline}\label{eq:b+I1}\left\|\B^{+}\left(\I^{0}_{t+h}[f]-\I^{0}_{t}[f]\right)\right\|_{\lp}^{p}=\int_{\Gamma_{+}}\left|\int_{t}^{t+h}f(\Phi(\z,-s))\chi_{\{s < \t_{-}(\z)\}}\d s\right|^{p}\d\mu_{+}(\z)\\
\leq h^{p-1}\int_{\Gamma_{+}}\int_{t}^{t+h}|f(\Phi(\z,-s))|^{p}\chi_{\{s < \t_{-}(\z)\}}\d s\end{multline}
where we used H\"older's inequality in the last inequality. One recognizes then thanks to Eq. \eqref{eq:f-UO}  that the last integral in the above inequality coincides with $\|U_{0}(t)f\|_{p}^{p}-\|U_{0}(t+h)f\|_{p}^{p}.$ This proves \eqref{eq:b+I}. 
\end{proof}
\begin{nb} Notice that, for $t=0$ the above \eqref{eq:b+I1} becomes
\begin{equation}\label{eq:b+I2}
\|\B^{+}\I^{0}_{h}[f]\|_{\lp}
\leq h^{1-1/p}\|f\|_{p},\qquad \forall h>0.
\end{equation}
\end{nb}
One also has
\begin{propo}\label{prop:B+U0}
For any $f \in \D_{0}$, any $t \geq 0$, the traces $\B^{\pm}U_{0}(t)f \in L^{p}_{\pm}$ and the mappings $t \geq 0 \mapsto \B^{\pm}U_{0}(t)f \in L^{p}_{\pm}$ are continuous with 
$$\int_{0}^{t}\|\B^{+}U_{0}(s)f\|_{\lp}^{p}\d s=\|f\|_{p}^{p}-\|U_{0}(t)f\|_{p}^{p}, \qquad \forall t \geq 0.$$
In particular, 
\begin{equation}\label{eq:b0?}
\int_{0}^{t}\|\B^{+}U_{0}(s)f\|_{\lp}^{p} \d s\leq \frac{t^{p}}{p}\int_{0}^{t}\left(\int_{\Gamma_{+}}\left|[\T_{\mathrm{max},\,p}f](\Phi(\z,-s))\right|^{p}
\chi_{\{s < \t_{-}(\z)\}}\d\mu_{+}(\z)\right)\d s.\end{equation}
\end{propo}
\begin{proof} For $f \in \D_{0}$, we have $U_{0}(t)f-f = \I^{0}_{t}[ \T_{0,\,p}f]$, and therefore $\B^{\pm}U_{0}(t)f=\B^{\pm} \I^{0}_{t}[ \T_{0,\,p}f]$. Thanks to Proposition \ref{prop:B+U01} we can state that the traces $\B^{\pm}U_{0}(t)f \in L^{p}_{\pm}$ and the mappings $t \geq 0 \mapsto \B^{\pm}U_{0}(t)f \in L^{p}_{\pm}$ are continuous. Then
$$\int_{0}^{t}\|\B^{+}U_{0}(s)f\|_{\lp}^{p}\d s=\int_{0}^{t}\left(\int_{\Gamma_{+}}|f(\Phi(\z,-s))|^{p}\chi_{\{s < \t_{-}(\z)\}}\d \mu_{+}(\z)\right)\d s$$
so that, thanks to Fubini's Theorem
$$\int_{0}^{t}\|\B^{+}U_{0}(s)f\|_{\lp}^{p}\d s=\int_{\Gamma_{+}}\left(\int_{0}^{t}|f(\Phi(\z,-s))|^{p}\chi_{\{s < \t_{-}(\z)\}}\d s\right)\d\mu_{+}(\z)$$
which, using again \eqref{eq:f-UO}, gives the first part of the Proposition. Let us now prove \eqref{eq:b0?}. Using Eq. \eqref{traceform}, one has from the previous identity that
\begin{multline*}
\int_{\Gamma_{+}}\left(\int_{0}^{t}|f(\Phi(\z,-s))|^{p}\chi_{\{s < \t_{-}(\z)\}}\d s\right)\d\mu_{+}(\z)\\
=\int_{\Gamma_{+}}\left(\int_{0}^{t}\left|\int_{0}^{s}\T_{\mathrm{max},\,p}f(\Phi(\z,-r))\d r\right|^{p}\chi_{\{s < \t_{-}(\z)\}}\d s\right)\d\mu_{+}(\z).
\end{multline*}
Then, since, for almost every $\z \in \Gamma_{+}$ and any $s \in (0,t)$:
\begin{multline*}
\left|\int_{0}^{s}\T_{\mathrm{max},\,p}f(\Phi(\z,-r))\d r\right|^{p} \leq s^{p-1}\int_{0}^{s}\left|\T_{\mathrm{max},\,p}f(\Phi(\z,-r))\right|^{p}\d r\\
\leq s^{p-1}\int_{0}^{t}\left|\T_{\mathrm{max},\,p}f(\Phi(\z,-r))\right|^{p}\d r\end{multline*}
we get the result after integrating with respect to $s$ over $(0,t)$.\end{proof}

We are now in position to define inductively the following:
\begin{defi}\label{defi:Uk}
Let $ t \geq 0$, $k \geq 1$ and $f \in \D_{0}$ be given. For $\x \in \O$ with $\t_{-}(\x) \leq t$, there exists a unique $\y \in \Gamma_{-}$ and a unique 
$0 < s < \min(t,\t_{+}(\y))$ such that $\x=\Phi(\y,s)$ 
and then one sets
$$[U_{k}(t)f](\x)=\left[H\B^{+}U_{k-1}(t-s)f\right](\y).$$
We set $[U_{k}(t)f](\x)=0$ if $\t_{-}(\x) > t$ and  $U_{k}(0)f=0$. 
\end{defi}
\begin{nb} Clearly, for $\x \in \O$ with $\t_{-}(\x) < t$, the unique $\y \in \Gamma_{-}$ and $s \in (0,\min(t,\t_{+}(\y))$ such that $\x=\Phi(\y,s)$ are 
$$\y=\Phi(\x,-\t_{-}(\x)), \qquad \quad s=\t_{-}(\x)$$
so that the above definition reads
$$[U_{k}(t)f](\x)=\left[H(\B^{+}U_{k-1}(t-\t_{-}(\x))f\right](\Phi(\x,-\t_{-}(\x))).$$
\end{nb}

The fact that, with this, $(U_{k}(t))_{t \geq 0}$ is a well-defined family which extends to a family of operators in $\mathscr{B}(X)$ satisfying the following is given in the Appendix.
\begin{theo}\label{theo:UKT} For any $k \geq 1$, $f \in \D_{0}$ one has $U_{k}(t)f \in X$ for any $t \geq 0$ with
$$\|U_{k}(t)f\|_{p} \leq \verti{H}^{k}\,\|f\|_{p}.$$
In particular, $U_{k}(t)$ can be extended to be a bounded linear operator, still denoted $U_{k}(t) \in\mathscr{B}(X)$ with
$$\|U_{k}(t)\|_{\mathscr{B}(X)} \leq \verti{H}^{k} \qquad \forall t \geq 0, k \geq 1.$$
Moreover, the following holds for any $k \geq 1$
\begin{enumerate}
\item $(U_{k}(t))_{t \geq 0}$ is a strongly continuous family of $\mathscr{B}(X)$.
\item For any $f \in X$ and any $t,s \geq 0$, it holds
$$U_{k}(t+s)f=\sum_{j=0}^{k}U_{j}(t)U_{k-j}(s)f.$$
\item For any $f \in \D_{0}$, the mapping $t \geq 0 \mapsto U_{k}(t)f$ is differentiable with 
$$\dfrac{\d}{\d t}U_{k}(t)f=U_{k}(t)\T_{\mathrm{max},\,p}f \qquad \forall t \geq 0.$$
\item For any $f \in \D_{0}$, one has $U_{k}(t)f \in \D(\T_{\mathrm{max},\,p})$ for all $t \geq 0$ with $\T_{\mathrm{max},\,p}U_{k}(t)f=U_{k}(t)\T_{\mathrm{max},\,p}f.$
\item For any $f \in X$ and any $t >0$, one has 
$$\I^{k}_{t}[f]:=\int_{0}^{t}U_{k}(s)f \d s \in \D(\T_{\mathrm{max},\,p}) \quad \text{ with } \quad  
\T_{\mathrm{max},\,p}\I_{t}^{k}[f]=U_{k}(t)f.$$
\item For any $f \in \D_{0}$ and any $t \geq 0$, the traces $\B^{\pm}U_{k}(t)f \in L^{p}_{\pm}$ and the mappings $t \geq 0\mapsto \B^{\pm}U_{k}(t)f \in L^{p}_{\pm}$ are continuous. Moreover, for all $f \in X$ and $t >0$, one has
$$\B^{\pm} \int_{0}^{t}U_{k}(s)f \d s \in L^{p}_{\pm} \qquad \text{ with } \quad \B^{-}\int_{0}^{t}U_{k}(s)f\d s=H\B^{+}\int_{0}^{t}U_{k-1}(s)f \d s.$$
\item For any $f \in \D_{0}$, it holds
$$\int_{0}^{t}\|\B^{+}U_{k}(s)f\|_{\lp}^{p}\d s \leq \verti{H}^{p}\,\int_{0}^{t}\|\B^{+}U_{k-1}(s)f\|_{\lp}^{p}\d s, \qquad \forall t \geq 0.$$
\item For any $f \in X$ and $\lambda >0$, setting $F_{k}=\int_{0}^{\infty}\exp(-\lambda t)U_{k}(t)f \d t$ one has
$$F_{k} \in \D(\T_{\mathrm{max},\,p}) \qquad \text{ with } \qquad \T_{\mathrm{max},\,p}F_{k}=\lambda\,F_{k}$$
and  $\B^{\pm}F_{k} \in L^{p}_{\pm}$ with 
$$\B^{-}F_{k}=H\B^{+}F_{k-1} \qquad \B^{+}F_{k}=(M_{\lambda}H)^{k}G_{\lambda}f.$$
\end{enumerate}
\end{theo}

\subsection{Generation Theorem} Introduce the following truncation operator

\begin{defi}
For any $\delta >0$, introduce 
$$\Gamma_{+}^{\delta}:=\left\{\z \in \Gamma_{+}\;;\;\t_{-}(\z) > \delta\right\}$$
and define the following truncation operator
$\chi_{\delta} \in \mathscr{B}(\lp)$ given by
$$\left[\chi_{\delta}\psi\right](\z)=\psi(\z)\chi_{\Gamma_{+}\setminus \Gamma_{+}^{\delta}}(\z), \qquad \forall \z \in \Gamma_{+}\,;\,\psi \in \lp.$$
\end{defi}

One has then the following
\begin{lemme}\label{lem:trunc} Assume that $H \in \mathscr{B}(\lp,\lm)$. 
Then, with the notations of Theorem \ref{theo:UKT},
\begin{multline}\label{eq:normb+uk}
\left(\int_{0}^{t}\|\B^{+}U_{k}(s)f\|_{\lp}^{p}\d s\right)^{1/p} \leq \sum_{j=0}^{\min(k,[t/\delta]+1)}\left(\begin{array}{c}k \\j\end{array}\right)\verti{H\chi_{\delta}}^{k-j}\verti{H}^{j}\,\|f\|_{p}\\
\text{ for any } f \in \D_{0},\quad \text{ and any } t >0.\end{multline}
while
\begin{equation}\label{eq:normUk}
\|U_{k}(t)\|_{\mathscr{B}(X)} \leq \sum_{j=0}^{\min(k,[t/\delta]+1)}\left(\begin{array}{c}k \\j\end{array}\right)\verti{H\chi_{\delta}}^{k-j}\verti{H}^{j}.\end{equation}
\end{lemme}
\begin{proof} For $k=0$, both inequalities are clearly true. Let $k \geq 1$ and $f \in \D_{0}, t >0$ be given. For $0 \leq s \leq t$, one has for $\mu_{+}$-a.e. $\z \in \Gamma_{+}$:
\begin{equation*}
\left[\B^{+}U_{k}(s)f\right](\z)=\left[H\left(\B^{+}U_{k-1}(s-\t_{-}(\z)\right)f\right](\Phi(\z,-\t_{-}(\z)))\chi_{(0,s)}(\t_{-}(\z))\end{equation*}
Thus, writing $H=H\chi_{\delta}+H(1-\chi_{\delta})$, we can estimate 
$$\left(\int_{0}^{t}\|\B^{+}U_{k}(s)f\|_{\lp}^{p}\d s\right)^{\frac{1}{p}} \leq  J_{1}+J_{2}$$ with
\begin{multline*}
J_{1}^{p}=\int_{0}^{t}\left(\int_{\Gamma_{+}}\left|\left[H\chi_{\delta}\left(\B^{+}U_{k-1}(s-\t_{-}(\z))\right)f\right](\Phi(\z,-\t_{-}(\z)))\right|^{p}\chi_{(0,s)}(\t_{-}(\z))\d \mu_{+}(\z)\right)\d s\\
J_{2}^{p}=\int_{0}^{t}\left(\int_{\Gamma_{+}}\left|\left[H(1-\chi_{\delta})\left(\B^{+}U_{k-1}(s-\t_{-}(\z))\right)f\right](\Phi(\z,-\t_{-}(\z)))\right|^{p}\chi_{(0,s)}(\t_{-}(\z))\d \mu_{+}(\z)\right)\d s.\end{multline*}
As in the proof of Lemma \ref{lem:C6}, one has
$$J_{1}^{p} \leq \verti{H\chi_{\delta}}^{p}\int_{0}^{t}\|\B^{+}U_{k-1}(s)f\|_{\lp}^{p}\d s$$
and 
\begin{multline*}
J_{2}^{p} \leq \int_{0}^{t}\left(\int_{\Gamma_{-}}\left|\left[H(1-\chi_{\delta})\left(\B^{+}U_{k-1}(s)\right)f\right](\y)\right|^{p}\d\mu_{-}(\y)\right)\d s\\
\leq 
\verti{H}^{p}\int_{0}^{t}\left(\int_{\Gamma_{+}^{\delta}}\left|\left[\B^{+}U_{k-1}(s)f\right](\z)\right|^{p}\d\mu_{+}(\z)\right)\d s\\
=\verti{H}^{p}\int_{0}^{t}\|\B^{+}U_{k-1}(s)f\|_{L^{p}(\Gamma_{+}^{\delta},\d\mu_{+})}^{p}\d s.\end{multline*}
Introduce now the following quantities, where we recall that $\delta >0$ is fixed: let $C_{\delta}=\verti{H\chi_{\delta}}$, $A=\verti{H}$ and, for any $k \geq 1$, 
$$S_{k}(t)=\left(\int_{0}^{t}\|\B^{+}U_{k}(s)f\|_{\lp}^{p}\d s\right)^{1/p}, \qquad Z_{k}(t)=\left(\int_{0}^{t}\|\B^{+}U_{k}(s)f\|_{L^{p}(\Gamma_{+}^{\delta},\d\mu_{+})}^{p}\d s\right)^{1/p}.$$
One proved already that
\begin{equation}\label{Eq:itera}
S_{k}(t) \leq C_{\delta}\,S_{k-1}(t) + A\,Z_{k-1}(t), \qquad \forall k \geq 1.\end{equation}
Let us now estimate inductively $Z_{k}(t)$. Assume $t > \delta$. One has, as before, splitting $H$ as $H=H\chi_{\delta}+H(1-\chi_{\delta})$:
$$Z_{k}(t) \leq \mathcal{J}_{1}+\mathcal{J}_{2}$$
with
\begin{equation*}\begin{split}
\mathcal{J}_{1}^{p}&=\int_{0}^{t}\left(\int_{\Gamma_{+}^{\delta}}\left|\left[H\chi_{\delta}\left(\B^{+}U_{k-1}(s-\t_{-}(\z))\right)f\right](\Phi(\z,-\t_{-}(\z)))\right|^{p}\chi_{(0,s)}(\t_{-}(\z))\d\mu_{+}(\z)\right)\d s\\
\mathcal{J}_{2}^{p}&=\int_{0}^{t}\left(\int_{\Gamma_{+}^{\delta}}\left|\left[H(1-\chi_{\delta})\left(\B^{+}U_{k-1}(s-\t_{-}(\z))\right)f\right](\Phi(\z,-\t_{-}(\z)))\right|^{p}\chi_{(0,s)}(\t_{-}(\z))\d\mu_{+}(\z)\right)\d s.\end{split}\end{equation*}
Now, as in the previous computation
\begin{equation*}\begin{split}
\mathcal{J}_{1}^{p}&=\int_{\Gamma_{+}^{\delta}}\left(\int_{0}^{\max(0,t-\t_{-}(\z))}\left|\left[H\chi_{\delta}\left(\B^{+}U_{k-1}(s)\right)f\right](\Phi(\z,-\t_{-}(\z)))\right|^{p}\d s \right)\d\mu_{+}(\z)\\
&\leq \int_{0}^{t-\delta}\left(\int_{\Gamma_{+}}\left|\left[H\chi_{\delta}\left(\B^{+}U_{k-1}(s)\right)f\right](\Phi(\z,-\t_{-}(\z)))\right|^{p}\d \mu_{+}(\z) \right)\d s\end{split}\end{equation*}
where we used that $\t_{-}(\cdot) >\delta$ on $\Gamma_{+}^{\delta}.$ We obtain then easily that
$$\mathcal{J}_{1}^{p} \leq C_{\delta}^{p}\,S_{k-1}^{p}(t-\delta).$$
In the same way $\mathcal{J}_{2}^{p} \leq A^{p}\,Z_{k-1}^{p}(t-\delta)$ 
which results in 
\begin{equation}\label{eq:itera1}
Z_{k}(t) \leq C_{\delta}\,S_{k-1}(t-\delta) + A\,Z_{k-1}(t-\delta), \qquad \forall t \geq \delta.\end{equation}
Combining this with \eqref{Eq:itera}, one obtains easily by induction that
\begin{equation*}
S_{k}(t) \leq \sum_{j=0}^{k-1}\left(\begin{array}{c}k-1 \\j\end{array}\right)C_{\delta}^{k-1-j}\,A^{j}\left(C_{\delta}\,S_{0}(t-j\delta) + A\,Z_{0}(t-j\delta)\right)\end{equation*}
and
\begin{equation*}
Z_{k}(t) \leq \sum_{j=0}^{k-1}\left(\begin{array}{c}k-1 \\j\end{array}\right)C_{\delta}^{k-1-j}\,A^{j}\left(C_{\delta}\,S_{0}(t-(j+1)\delta) + A\,Z_{0}(t-(j+1)\delta)\right)\end{equation*}
with the convention that $S_{k}(r)=Z_{k}(r)=0$ for $r <0$. Since $Z_{0}(t) \leq S_{0}(t) \leq \|f\|_{p}$ (see Prop. \ref{prop:B+U0}) and setting $k_{\delta}(t)=\min(k-1,[\frac{t}{\delta}])$ we get
$$S_{k}(t) \leq \|f\|_{p}\sum_{j=0}^{k_{\delta}(t)}\left(\begin{array}{c}k-1 \\j\end{array}\right)C_{\delta}^{k-1-j}\,A^{j}(C_{\delta}+A)$$
since $Z_{0}(t-j\delta)=S_{0}(t-j\delta)=0$ for $j \geq t/\delta$. Now, it is not difficult to check that
$$\sum_{j=0}^{k_{\delta}(t)}\left(\begin{array}{c}k-1 \\j\end{array}\right)C_{\delta}^{k-1-j}\,A^{j}\left(C_{\delta}+A\right) \leq \sum_{j=0}^{k_{\delta}(t)+1}\left(\begin{array}{c}k \\j\end{array}\right)C_{\delta}^{k-j}\,A^{j}$$ which gives \eqref{eq:normb+uk}.   Now, from the definition of $U_{k}(t)$, one has
\begin{equation*}
\|U_{k}(t)f\|_{p}^{p}=\int_{\Gamma_{-}}\left(\int_{0}^{\t_{+}(\y)}\left|\left[H\left(\B^{+}U_{k-1}(t-s)\right)f\right](\y)\right|^{p}\d s \right)\d\mu_{-}(\y)\end{equation*}
and, writing again $H=H\chi_{\delta}+H(1-\chi_{\delta})$ one arrives without difficulty to
$$\|U_{k}(t)f\|_{p} \leq C_{\delta} S_{k-1}(t) + AZ_{k-1}(t)$$
and, as before, this gives 
$$\|U_{k}(t)f\|_{p} \leq \displaystyle\sum_{j=0}^{\min(k,[t/\delta]+1)}\left(\begin{array}{c}k \\j\end{array}\right)\verti{H\chi_{\delta}}^{k-j}\verti{H}^{j}\,\|f\|_{p}$$
for any $f \in \mathscr{D}_{0}$ and we obtain the result by density.\end{proof}

This allows to prove the following
\begin{theo}
Assume that $H \in \mathscr{B}(\lp,\lm)$ is such that
\begin{equation}\label{eq:trunc}
\limsup_{\delta\to0^{+}}\verti{H\chi_{\delta}} < 1.\end{equation}
Then, the series $\sum_{k=0}^{\infty}\|U_{k}(t)\|_{\mathscr{B}(X)}$ is convergent for any $t \geq 0$ and, setting 
$$U_{H}(t)=\sum_{k=0}^{\infty}U_{k}(t), \qquad t \geq 0$$
it holds that $(U_{H}(t))_{t\geq 0}$ is a $C_{0}$-semigroup in $X$ with generator $\left(\T_{H,\,p},\D(\T_{H,\,p})\right).$
\end{theo}
\begin{proof} Let $\delta_{0} >0$ be such that $C:=\sup_{\delta \in (0,\delta_{0})}\verti{H\chi_{\delta}} < 1$ and let us consider from now on $\delta < \delta_{0}.$ Fix $t \geq 0$ and, with the notations of the above Lemma, let 
$$\bm{u}_{k}=\sum_{j=0}^{\min(k,[t/\delta]+1)} \left(\begin{array}{c}k \\j\end{array}\right)\verti{H\chi_{\delta}}^{k-j}\verti{H}^{j}, \qquad k \geq 0.$$
The series $\sum_{k}\bm{u}_{k}$ is convergent. Indeed, setting $A:=\verti{H}$, one has
$$\sum_{k=0}^{\infty}\bm{u}_{k} \leq \sum_{j=0}^{[t/\delta]+1}A^{j}\sum_{k=j}^{\infty}\left(\begin{array}{c}k \\j\end{array}\right)C^{k-j}.$$
Using the well-known identity, valid for $0\leq C< 1$:
$$\sum_{k=j}^{\infty}\left(\begin{array}{c}k \\j\end{array}\right)C^{k-j}=\dfrac{1}{(1-C)^{j+1}}$$
we obtain that 
\begin{equation}\label{eq:ukconv}
\sum_{k=0}^{\infty}\bm{u}_{k} \leq \frac{1}{1-C}\sum_{j=0}^{[t/\delta]+1}\left(\frac{A}{1-C}\right)^{j} \leq M\exp(\omega\,t)\end{equation}
with
\begin{align}\label{eq:Momega}
M&=\frac{A^2}{(1-C)^{2}(A+C-1)} \quad &\text{ and }  \quad \omega&=\frac{1}{\delta} \log\left(\frac{A}{1-C}\right) \qquad &\text{ if } \quad A > 1-C,\nonumber\\
M&=\frac{2}{1-C} \quad &\text{ and } \quad \omega&=\frac{1}{\delta} \qquad &\text{ if } \quad A=1-C\nonumber\\
M&=\frac{1}{1-A-C} \quad &\text{ and }  \quad \omega&=0 \qquad  &\text{ if }\quad A < 1-C  \end{align}
\medskip

According to Lemma \ref{lem:trunc}, one sees that, for any $\delta \in (0,\delta_{0})$, 
$$\sum_{k=0}^{\infty}\left\|U_{k}(t)\right\|_{\mathscr{B}(X)} \leq M\exp(\omega\,t), \qquad \forall t \geq 0.$$
This proves that, for any $t \geq 0$, the series $\sum_{k=0}^{\infty}U_{k}(t)$ converges in $\mathscr{B}(X)$ and denote its sum by $U_{H}(t)$. One has
$$\|U_{H}(t)\|_{\mathscr{B}(X)} \leq M\exp(\omega\,t), \qquad t \geq 0$$
and, according to Theorem \ref{theo:UKT}, $(U_{H}(t))_{t\geq 0}$ is a strongly continuous family of $\mathscr{B}(X)$ with 
$$\lim_{t\to 0^{+}}U_{H}(t)f=\lim_{t\to 0^{+}}U_{0}(t)f=f$$ for any $f \in X$. Finally, using point (2) of Theorem \ref{theo:UKT}, one sees that $(U_{H}(t))_{t\geq 0}$ is a $C_{0}$-semigroup in $X$. Let us denote by $\mathcal{A}$ its generator. We prove exactly as in \cite[Theorem 4.1]{arlo11} that $\mathcal{A}=\T_{H,\,p}$.

\noindent \emph{First Step: $\T_{H,\,p}$ is an extension of $\A$.} Let $g \in \D(\A)$ and $\lambda >\omega$ be given. Set $f=(\lambda-\A)g$. As known, and using the notations of Theorem \ref{theo:UKT}:
$$g=\int_{0}^{\infty}\exp(-\lambda\,t)U_{H}(t)f\d t=\sum_{k=0}^{\infty}\int_{0}^{\infty}\exp(-\lambda\,t)U_{k}(t)f\d t=\sum_{k=0}^{\infty}F_{k}.$$
According to \textit{(8)} in Theorem \ref{theo:UKT}, one has $g \in \D(\T_{\mathrm{max},\,p})$ with $\T_{\mathrm{max},\,p}g=\lambda\,g-f$, i.e. $\T_{\mathrm{max},\,p}g=\A\,g$. Moreover, $\B^{+}g=\B^{+}(\sum_{k=0}^{\infty}F_{k})$. By virtue of \textit{(8)} of Theorem \ref{theo:UKT},
\begin{equation}\begin{split}\label{eq:b+fkco}
\|\B^{+}F_{k}\|_{\lp}&=\|\left(M_{\lambda}H\right)^{k}G_{\lambda}f\|_{\lp}=\|M_{\lambda}(HM_{\lambda})^{k-1}HG_{\lambda}f\|_{\lp}\\
& \leq \|(HM_{\lambda})^{k-1}HG_{\lambda}f\|_{\lm}\end{split}\end{equation}
Now,
\begin{equation*}\begin{split}
\|HM_{\lambda}\|_{\mathscr{B}(\lm)} &\leq \|H\chi_{\delta}M_{\lambda}\|_{\mathscr{B}(\lm)} + \|H(1-\chi_{\delta})M_{\lambda}\|_{\mathscr{B}(\lm)}\\
&\leq \|H\chi_{\delta}\|_{\mathscr{B}(\lp,\lm)}+ \verti{H}\exp(-\lambda\delta)\end{split}\end{equation*}
where we used that, as in Eq. \eqref{eq:mllm}, for all $u \in \lm$
$$\|(1-\chi_{\delta})M_{\lambda}\|_{\lp}^{p}=\int_{\Gamma_{+}^{\delta}}|[M_{\lambda}u](\z)|^{p}\d\mu_{+}(\z) \leq \exp(-p\lambda\delta)\|u\|_{\lm}^{p}.$$
In other words, with the above notations, for $\lambda > \delta$
\begin{equation}\label{eq:HML}
\|HM_{\lambda}\|_{\mathscr{B}(\lm)}  \leq C + A \exp(-\lambda\delta) < C+A\exp(-\omega\delta) \leq1\end{equation}
by definition of $\omega$ (see Eq. \eqref{eq:Momega}). This, together with \eqref{eq:b+fkco} shows that the series
$\sum_{k=0}^{\infty}\|\B^{+}F_{k}\|_{\lp}$ converges and therefore, 
$$\B^{+}g=\sum_{k=0}^{\infty}\B^{+}F_{k} \in \lp$$
and, being $H$ continuous, using again \textit{(8)} of Theorem \ref{theo:UKT}
$$\B^{-}g=\sum_{k=0}^{\infty}H\B^{+}F_{k}=H\B^{+}g.$$
This proves that $g \in \D(\T_{H,\,p})$, i.e.
$$\D(\A) \subset \D(\T_{H,\,p}) \qquad \text{ and } \quad \A\,g=\T_{\mathrm{max},\,p}g=\T_{H,\,p}g, \qquad \forall g \in \D(\A).$$

\noindent \emph{Second step: $\A$ is an extension of $\T_{H,\,p}$.} Conversely, let $g \in \D(\T_{H,\,p})$ and $\lambda >\omega$ be given. Set $f=(\lambda-\T_{H,\,p})g$. We define 
$$F=\int_{0}^{\infty}\exp(-\lambda\,t)U_{H}(t)f\d t=(\lambda-\A)^{-1}f, \qquad \text{ and } \quad G=g-F.$$
From the first point, $G \in \D(\T_{H,\,p})$ with $\T_{H,\,p}G=\lambda\,G$. In particular, $G=\Xi_{\lambda}\B^{-}G$ while
$$\B^{-}G=H\B^{+}G=HM_{\lambda}\B^{-}G.$$
But, from \eqref{eq:HML}, one has $\|\B^{-}G\|_{\lm} < \|\B^{-}G\|_{\lm}$, i.e. $\B^{-}G=0.$ Since $G=\Xi_{\lambda}\B^{-}G$, we get $G=0$ and $g=F$ and this proves $\A=\T_{H,\,p}.$
\end{proof}
\begin{nb} Notice that \eqref{eq:HML} with Theorem \ref{TheoDautray} already ensured that $(\T_{H,\,p},\D(\T_{H,\,p}))$ generates a $C_{0}$-semigroup in $X$. The interest of the above lies of course in the fact that it allows to deal with multiplicative boundary operator for which $\verti{H}\geq 1.$ The novelty of the above approach, with respect to \cite[Theorem 5.1]{mjm2}, is that the above proof is \emph{constructive} and we give a precise and explicit expression of the semigroup $(U_{H}(t))_{t\geq 0}.$ The proof presented here is similar to the one in \cite{arlo11} and differs from the one in our previous contribution \cite{mjm2}. Notice however that it would be possible to adapt in a simple way the proof given in \cite[Section 5]{mjm2} which is based on some suitable change of variables.\end{nb}

\begin{nb} The above estimate \eqref{eq:ukconv} with $M,\omega$ given by \eqref{eq:Momega} allowed us to prove the convergence (in $\mathscr{B}(X)$) of the series $\sum_{k=0}^{\infty}U_{k}(t)$ but does not yield the optimal estimate for the limit $\|U_{H}(t)\|_{\mathscr{B}(X)}.$ For instance, if $A=\verti{H} < 1$ then the semigroup $(U_{H}(t))_{t \geq 0}$ is a contraction semigroup while the above estimate yields $\|U_{H}(t)\|_{\mathscr{B}(X)} \leq \frac{1}{1-A-C}$ with $\frac{1}{1-A-C} >1$.\end{nb}

 An useful consequence of the above is the following more tractable expression of the semigroup $(U_{H}(t))_{t\geq0}$:
\begin{cor}\label{cor:semi} For any $f \in \D(\T_{H,\,p})$ and any $t \geq 0$, the following holds for $\mu$-a.e. $\x \in \O$:
$$U_{H}(t)f(\x)=\begin{cases} U_{0}(t)f(\x)=f(\Phi(\x,-t)) \qquad &\text{ if } t < \t_{-}(\x)\\
\left[H\left(\B^{+}U_{H}(t-\t_{-}(\x)) f \right )\right](\Phi(\x,-\t_{-}(\x))) \qquad &\text{ if } t \geq \t_{-}(\x).\end{cases}
$$\end{cor}
\begin{nb} Notice that, if $f \in \D(T_{H,\,p})$, for any $t \geq 0$, $\psi(t)=U_{H}(t)f$ is the unique classical solution (see \cite{pazy}) to the Cauchy problem
$$\dfrac{\d}{\d t}\psi(t)=\T_{H,\,p}\psi(t), \qquad \psi(0)=f.$$
The above provides therefore the (semi)-explicit expression of the solution to the above.\end{nb}
\begin{proof}
Let us consider first $f \in \D_{0}$. Then, for all $k \geq 0$, $t >0$
$$\B^{+}U_{k}(t)f=\B^{+}\left(\int_{0}^{t}U_{k}(s)\T_{\mathrm{max},\,p}f\d s\right).$$
In particular, from the previous Theorem, the series $\sum_{k=0}^{\infty}\|\B^{+}U_{k}(t)f\|_{\lp}$ is convergent and therefore 
$$\B^{+}U_{H}(t)f=\sum_{k=0}^{\infty}\B^{+}U_{k}(t)f$$
where the series converge in $\lp$. Given $0 \leq s < t$, we get then 
$$HB^{+}U_{H}(t-s)f=\sum_{k=0}^{\infty}H\B^{+}U_{k}(t-s)f.$$
Pick then $\x \in \O$ and $\y \in \Gamma_{-}$ such that $\x=\Phi(\y,s)$. We have, by definition, 
$$\left[H\B^{+}U_{k}(t-s)f\right](\y)=[U_{k+1}(t)f](\x).$$
Therefore, 
$$[H\B^{+}U_{H}(t-s)f](\y)=\sum_{k=1}^{\infty}\left[U_{k+1}(t)f\right](\x)=[U_{H}(t)f](\x)-[U_{0}(t)f](\x).$$
To summarize, for almost every $\x \in \O$, there exists a unique $\y \in \Gamma_{-}$ and a unique 
$0 < s < \t_{+}(\y)$ such that $\x=\Phi(\y,s)$ and  we proved that
$$[U_{H}(t)f](\x)=[U_{0}(t)f](\x) + [H\B^{+}U_{H}(t-s)f](\y)$$
which proves the result for $f \in \D_{0}.$

Consider now $f \in \D(\T_{H,\,p})$. Then, for any $t >0$, $U_{H}(t)f \in \D(\T_{H,\,p})$. Introduce then the mapping
$$\x \in \O_{-} \longmapsto \left[H\B^{+}U_{H}(t-s)f\right](\y)=g(t,\x)$$
where $\x=\Phi(\y,s)$ for some unique $\y \in \Gamma_{-}$ and $s \in (0,\t_{+}(\y)).$ It holds, for any $\lambda >\omega$ (with $\omega$ introduced in the proof of the previous Theorem):
\begin{multline*}
\int_{0}^{\infty}\exp(-\lambda\,t)g(t,\x)\d t=\int_{0}^{\infty}\exp(-\lambda\,t)\left[H\B^{+}U_{H}(t-s)f\right](\y)\d t\\
=\exp(-\lambda\,s)\int_{0}^{\infty}\exp(-\lambda\,t)\left[H\B^{+}U_{H}(t)f\right](\y)\d t\\=\exp(-\lambda\,s)\left[H\B^{+}\int_{0}^{\infty}\exp(-\lambda\,t)U_{H}(t)f\d t\right](\y)\end{multline*}
Therefore, using point \textit{(8)} of Theorem \ref{theo:UKT}
\begin{multline*}
\int_{0}^{\infty}\exp(-\lambda\,t)g(t,\x)\d t=\exp(-\lambda\,s)\left[H\sum_{k=0}^{\infty}\left(M_{\lambda}H\right)^{k}G_{\lambda}f\right](\y)\\
=\exp(-\lambda\,s)\left[H(I-M_{\lambda}H)^{-1}G_{\lambda}f\right](\y).\end{multline*}
Since moreover
$$\int_{0}^{\infty}\exp(-\lambda\,t)U_{H}(t)f\d t=\int_{0}^{\infty}\exp(-\lambda\,t)U_{0}(t)f\d t + \Xi_{\lambda}H(I-M_{\lambda}H)^{-1}G_{\lambda}f$$
the result follows.
\end{proof}

\appendix

\section{Proof  of Theorem \ref{representation}}\label{app:techn}

The scope here is to prove Theorem \ref{representation} in Section \ref{sec:maxi}. The difficult
 part of the proof is the  implication $(2) \implies (1)$. It is carried out through
several technical lemmas based
 upon \textit{\textbf{mollification along the characteristic curves}}  (recall
 that, whenever $\mu$ is not absolutely continuous with respect to
 the Lebesgue measure, no global convolution argument is available). Let us make precise what this is all about. Consider
a sequence $(\varrho_n)_n$  of one dimensional mollifiers supported
in $[0,1]$, i.e. for any $n \in \mathbb{N}$, $\varrho_n \in
\Con^{\infty}_0(\mathbb{R})$, $\varrho_n(s)=0$ if $s \notin
[0,1/n]$, $\varrho_n(s) \geq 0$  and $\int_0^{1/n}\varrho_n(s)\d
s=1.$ Then, for any $f \in L^1(\O,\d\mu)$, define the (extended)  mollification:
$$\varrho_n \diamond f(\x)=\int_0^{\t_-(\x)}
\varrho_n(s)f({{\Phi}}(\x,-s))\d s.$$

Note that, with such a
definition, it is not clear {\it a priori} that $\varrho_n \diamond
f$ defines a measurable function, finite almost everywhere. It is
proved in the following that actually such a function does actually belong to $L^{p}(\O,\d\mu)$.
\begin{lemme} Given $f \in L^p(\O,\d\mu)$, $\varrho_n \diamond f \in
  L^p(\O ,\d\mu)$ for any $n \in \mathbb{N}$. Moreover,
\begin{equation}\label{contrac}
\|\varrho_n \diamond f\|_{p} \leq \|f\|_{p},\qquad \forall f \in
L^p(\O,\d\mu), n \in \mathbb{N}.\end{equation}
\end{lemme}
\begin{proof} One considers, for a given $f \in
L^p(\O,\d\mu)$,  the extension of $f$ by zero outside $\O$:
$$\overline{f}(\x)=f(\x),\qquad \forall \x \in \O,\qquad
\overline{f}(\x)=0 \quad \forall \x \in \mathbb{R}^N \setminus \O.$$
Then $\overline{f} \in L^p(\mathbb{R}^N,\d\mu).$ Let us consider the
 transformation:
$$\Upsilon\::\:(\x,s) \in \mathbb{R}^N \times \mathbb{R} \mapsto
\Upsilon(\x,s)=({{\Phi}}(\x,-s),-s) \in \mathbb{R}^N \times
\mathbb{R}.$$ As a homeomorphism, $\Upsilon$ is measure preserving
for  pure Borel measures. It is also measure preserving for
completions of Borel measures (such as a Lebesgue measure) since
it is measure-preserving on Borel sets and the completion of a
measure is obtained by adding to the Borel $\sigma$-algebra all
sets contained in measure-zero Borel sets, see \cite[Theorem
13.B, p.~55]{Hal}. Then, according to \cite[Theorem 39.B,
p.~162]{Hal}, the mapping
$$(\x,s) \in \mathbb{R}^N \times \mathbb{R} \mapsto
\overline{f}({{\Phi}}(\x,-s))$$ is measurable as the
composition of $\Upsilon$ with the measurable function $(\x,s)
\mapsto \overline{f}(\x)$. Define now $\Lambda=\{(\x,s)\,;\,\x \in
\O,\,0 < s < \t_-(\x)\},$ $\Lambda$ is a measurable subset of
$\mathbb{R}^N \times \mathbb{R}$. Therefore, the mapping
$$(\x,s) \in \mathbb{R}^N \times \mathbb{R} \longmapsto
\overline{f}({{\Phi}}(\x,-s))\chi_{\Lambda}(\x,s)
\varrho_n(s)$$ is measurable. In the same way
$$(\x,s) \in \mathbb{R}^N \times \mathbb{R} \longmapsto
\left|\overline{f}({{\Phi}}(\x,-s))\right|^{p}\chi_{\Lambda}(\x,s)
\varrho_n(s)$$ is measurable. For almost every $\x \in \O$, it holds
$$\int_{0}^{\infty}\varrho_{n}(s)\left|\overline{f}({{\Phi}}(\x,-s))\right|^{p}\chi_{\Lambda}(\x,s)\d s=\int_{0}^{\min(\tau_{-}(\x),1/n)}\varrho_{n}(s)\left|{f}({{\Phi}}(\x,-s))\right|^{p}\d s.$$
Setting $q=\frac{p}{p-1}$, observe now that $\varrho_{n}^{1/q} \in L^{q}(0,1/n)$ while, for a.e. $\x\in \O$, 
$$s \mapsto \varrho_{n}(s)^{1/p}f({\Phi}(\x,-s)) \in L^{p}(0,\min(\tau_{-}(\x),1/n)).$$  Therefore, for almost every $\x \in \O$
$$\varrho_{n}(s)f({\Phi}(\x,-s)) \in L^{1}(0,\min(\tau_{-}(\x),1/n))$$
thanks to Holder's inequality. Thus,
\begin{equation*}
[\varrho_n \diamond f](\x):= \ds
\int_0^{\tau_-(\x)}\varrho_n(s)f({{\Phi}}(\x,-s))\d
s\end{equation*} is finite for almost every $\x \in \O$ with
\begin{multline*}
\left|[\varrho_n \diamond f](\x)\right| \leq \left(\int_{0}^{1/n}\varrho_{n}(s)\d s\right)^{1/q}\,\left(\int_{0}^{\tau_{-}(\x)}\varrho_{n}(s)\left|{f}({{\Phi}}(\x,-s))\right|^{p}\d s\right)^{1/p}\\
=\left(\int_{0}^{\tau_{-}(\x)}\varrho_{n}(s)\left|{f}({{\Phi}}(\x,-s))\right|^{p}\d s\right)^{1/p}\end{multline*}
From this one sees that
\begin{equation}\label{eq:convo}
\left|[\varrho_n \diamond f](\x)\right|^{p} \leq [\varrho_{n}\diamond |f|^{p}](\x).\end{equation}
Now, since $|f|^{p} \in L^{1}(\O,\d\mu)$, one can use \cite[Theorem 3.7]{mjm1} to get first that $\varrho_{n}\diamond |f|^{p} \in L^{1}(\O,\d\mu)$ with
$$\|[\varrho_n \diamond |f|^{p}\|_{1} \leq \|\,|f|^{p}\|_{1}=\|f\|_{p}^{p}.$$
One deduces from this and \eqref{eq:convo} that  $[\varrho_n \diamond f] \in L^{p}(\O,\d\mu)$ and \eqref{contrac} holds true.\end{proof}

As it is the case for classical convolution, the family $(\varrho_n
\diamond f)_n$ approximates $f$ in $L^p$-norm:
\begin{propo}\label{approx} Given $f \in L^p(\O,\d\mu)$,
\begin{equation}\label{convergconv}
\lim_{n \to \infty}\int_{\O }\bigg|(\varrho_n \diamond f)(\x)
-f(\x)\bigg|^{p}\d\mu(\x)=0.\end{equation}
\end{propo}
\begin{proof} The proof is very similar to that of \cite[Proposition 3.8]{mjm1}. Let us fix a nonnegative $f$ continuous over $\O$ and
compactly supported. We introduce, for any $n \in \mathbb{N}$, $\mathcal{O}_{n}:=\mathrm{supp}(\varrho_n \diamond f) \cup \mathrm{supp}(f)$ and 
$\mathcal{O}_{n}^{-}=\{\x \in \mathcal{O}_{n}\;;\;\t_{-}(\x) <1/n\}.$  Since  $\sup_{\x \in \O}
|\varrho_n \diamond f(\x)| \leq \sup_{\x \in \O} |f(\x)|,$ for any
$\varepsilon > 0$, there exists $n_0 \geq 1$ such that
$$
\int_{\mathcal{O}_n^-}|f(\x)|^{p}\d\mu(\x) \leq \varepsilon, \quad \text{ and }
\quad \int_{\mathcal{O}_n^-}|\varrho_n \diamond f(\x)|^{p}\d\mu(\x) \leq
\varepsilon \qquad \forall n \geq n_0.$$
Now, noticing that $
\mathrm{Supp}(\varrho_n \diamond f -f) \subset \mathcal{O}_n $, one
has for any $n \geq n_0$,
\begin{equation}
\label{eq:on-}\int_{\O}|\varrho_n \diamond f-f|^{p}\d\mu=\int_{\mathcal{O}_n}|\varrho_n
\diamond f-f|^{p}d\mu \leq 2\varepsilon + \int_{\mathcal{O}_n \setminus
 \mathcal{O}_n^-}|\varrho_n \diamond f-f|^{p}\d\mu.\end{equation}
For any $\x \in \mathcal{O}_n \setminus \mathcal{O}_n^-$, since
$\varrho$ is supported in $[0,1/n]$, one has
\begin{equation*}\begin{split}
[\varrho_n \diamond f](\x)-f(\x) &=\int_0^{1/n}\varrho_n(s)
 f({{\Phi}}(\x, -s))\d
s-f(\x)\\
&=\int_0^{1/n}\varrho_n(s)\left(f({{\Phi}}(\x,
-s))-f(\x)\right)\d s.\end{split}\end{equation*} 
Then, as in the previous Lemma, one deduces from Holder inequality that
$$\left|[\varrho_n \diamond f](\x)-f(\x)\right|^{p} \leq \int_{0}^{1/n}\varrho_{n}(s)\left|f({\Phi}(\x,-s)-f(\x)\right|^{p}\d s.$$
As in the proof of \cite[Proposition 3.8]{mjm1}, one sees that, because $f$ is uniformly continuous on $\mathcal{O}_{1}$, there exists some $n_1 \geq 0$, such that 
$$\underset{\x \in \mathcal{O}_{1}}{\sup_{s \in (0,1/n)}}\left|f({\Phi}(\x,-s)-f(\x)\right|^{p} \leq \varepsilon \qquad \forall n \geq n_{1}$$
which results in 
$\left|[\varrho_n \diamond f](\x)-f(\x)\right|^{p} \leq \varepsilon$ for any $\x \in \mathcal{O}_n
\setminus \mathcal{O}_n^-$ and any $n \geq n_1$. One obtains then,
for any $n \geq n_1$,
$$\int_{\O}|\varrho_n \diamond f-f|^{p}\d\mu \leq 2\varepsilon +\varepsilon
\mu(\mathcal{O}_n \setminus \mathcal{O}_n^-) \leq 2\varepsilon
+\varepsilon \mu(\mathcal{O}_1)$$ which proves the
result.\end{proof}
As in \cite[Lemma 3.9, Proposition 3.11]{mjm1}, one has the following
\begin{lemme}\label{moll} Given $f \in L^p(\O,\d\mu)$, set $f_n=\varrho_n \diamond f$, $n \in \mathbb{N}$.
Then, $f_n \in \D(\T_{\mathrm{max},\,p})$ with
\begin{equation}
[\T_{\mathrm{max},\,p}
f_n](\x)=-\int_0^{\t_-(\x)}\varrho_n'(s)f({{\Phi}}(\x,-s))\d
s, \qquad \x \in \O.\end{equation}

Moreover,  for $f \in \D(\T_{\mathrm{max},\,p}  )$, then
\begin{equation}\label{claimdiamond}
[\T_{\mathrm{max},\,p}  (\varrho_n \diamond f)](\x)=[\varrho_n \diamond
\T_{\mathrm{max},\,p} f](\x),\qquad \qquad (\x \in \O\,,\,n \in
\mathbb{N}).\end{equation}
\end{lemme}

We are in position to prove the following
\begin{propo}\label{prop312} Let $f \in L^p(\O,\d\mu)$ and $f_n=\varrho_n \diamond f$, $n \in \mathbb{N}$. Then, for $\mu_-$-- a. e.
$\y \in \Gamma_-$,
\begin{equation}\label{fny-}
f_n({{\Phi}}(\y,s))-f_n({{\Phi}}(\y,t))=\int_s^t
[\T_{\mathrm{max},\,p} f_n]({{\Phi}}(\y,r))\d r \qquad \forall 0 <
s < t < \t_+(\y).\end{equation} In the same way, for almost every
$\z \in \Gamma_+$,
\begin{equation}\label{fny+}
f_n({{\Phi}}(\z,-s))-f_n({{\Phi}}(\z,-t))=-\int_s^t [\T_{\mathrm{max},\,p}
f_n]({{\Phi}}(\z,-r))\d r, \qquad \forall 0 < s < t <
\t_-(\z).\end{equation}
Moreover, for any $f \in \D(\T_{\mathrm{max},\,p})$, there exist some functions
$\widetilde{f}_{\pm} \in L^p(\O_\pm,\d\mu)$ such that
$\widetilde{f}_\pm(\x)=f(\x)$ for $\mu$- almost every $\x \in
\O_\pm$ and, for $\mu_-$--almost every $\y \in \Gamma_-$:
\begin{equation}\label{fy-}
\widetilde{f }_-({{\Phi}}(\y,s))-\widetilde{f
}_-({{\Phi}}(\y,t))=\int_s^t [\T_{\mathrm{max},\,p} f]
({{\Phi}}(\y,r))\d r \qquad \forall 0 < s < t <
\t_+(\y),\end{equation} while, for $\mu_+$--almost every $\z \in
\Gamma_+$:
\begin{equation}\label{fy+}
\widetilde{f }_+({{\Phi}}(\z,-s))-\widetilde{f
}_+({{\Phi}}(\z,-t))=-\int_s^t [\T_{\mathrm{max},\,p} f ]
({{\Phi}}(\z,-r))\d r \qquad \forall 0 < s < t < \t_-(\z).
\end{equation}
\end{propo}
 
\begin{proof} The proof is very similar to that of \cite[Propositions 3.13 \& 3.14]{mjm1}. Because $f \in L^{p}(\O,\d\mu)$, the integral
$$\int_{0}^{\t_+(\y)}|f({{\Phi}}(\y,r))|^{p}\d r$$
exists and is finite for  $\mu_-$-almost every $\y \in \Gamma_-$. As such, for
$\mu_-$-almost every $\y \in \Gamma_-$ and  any $0 < t < \t_+(\y)$, one has
\begin{multline*}
\left|\int_{0}^{t}\varrho_{n}(t-s)f({\Phi}(\y,s))\d s\right| \leq \int_{0}^{t}\varrho_{n}(t-s)|f({\Phi}(\y,s))|\d s\\
\leq \left(\int_{0}^{t}\varrho_{n}(t-s)\d s\right)^{1/q}\,\left(\int_{0}^{t}\varrho_{n}(t-s)|f({\Phi}(\y,s))|^{p}\d s\right)^{1/p} < \infty\end{multline*}
This shows that, for $\mu_{-}$-almost every $\y \in \Gamma_{-}$ and any $0<t<\t_{+}(\y)$, the quantity $\int_{0}^{t}\varrho_{n}(t-s)f({\Phi}(\y,s))\d s$ is well-defined and finite. The same argument shows that also $\int_0^t \varrho_n'(t-s) f({{\Phi}}(\y,s))\d s$ is
well-defined and finite for $\mu_{-}$-almost every $\y \in \Gamma_{-}$ and any $0<t<\t_{+}(\y)$. The rest of the proof is exactly as in \cite[Proposition 3.13]{mjm1} by virtue of Lemma \ref{moll} yielding to \eqref{fny-}--\eqref{fny+}. 

The proof of \eqref{fy-}--\eqref{fy+} is deduced from \cite[Proposition 3.14]{mjm1}. Namely, for any $n \geq 1$, set $f_n=\varrho_n \diamond
f$, so that, from Proposition  \ref{approx} and \eqref{claimdiamond},
$\lim_{n \to \infty}\|f_n -f\|_{p}  + \|\T_{\mathrm{max},\,p} f_n
-\T_{\mathrm{max},\,p} f\|_{p} =0.$ In particular,  from
Eq. \eqref{10.47}
\begin{multline*}
\int_{\Gamma_-}\d\mu_-(\y)\int_0^{\t_+(\y)}\left|f_n({{\Phi}}(\y,s))-f({{\Phi}}(\y,s))\right|^{p}\d
s \\
+ \int_{\Gamma_-}\d\mu_-(\y)\int_0^{\t_+(\y)}\left|[\T_{\mathrm{max},\,p}
f_n]({\Phi}(\y,s))- [\T_{\mathrm{max},\,p}
f]({{\Phi}}(\y,s))\right|^{p}\d s \underset{n\to
\infty}{\longrightarrow} 0\end{multline*} since $\T_{\mathrm{max},\,p} f$
and $\T_{\mathrm{max},\,p} f_n$ both belong to $L^p(\O,\d\mu)$.
Consequently, for almost every $\y \in \Gamma_-$ (up to a
subsequence, still denoted by $f_n$) we get
\begin{equation}\label{eq:conv}
\begin{cases}
f_n({{\Phi}}(\y,\cdot)) \longrightarrow f({{\Phi}}(\y,\cdot))\\
[\T_{\mathrm{max},\,p} f_n ]({{\Phi}}(\y,\cdot )) \longrightarrow
[\T_{\mathrm{max},\,p} f ]({{\Phi}}(\y,\cdot )) \quad \text{ in }
\quad L^p((0,\t_+(\y))\,,\d s)\end{cases}\end{equation} as $n \to
\infty$. Let us fix $\y \in \Gamma_-$ for which this holds. Passing
again to a subsequence, we may assume that $f_n({{\Phi}}(\y,s
))$ converges (pointwise) to $f({{\Phi}}(\y,s ))$ for almost
every $s \in (0,\t_+(\y))$. Let us fix such a $s_0$. From \eqref{fny-} 
\begin{equation}\label{eq:fnTfn}
f_n({{\Phi}}(\y,s_0))-f_n({{\Phi}}(\y,s ))=\int_{s_0}^s[\T_{\mathrm{max},\,p} f_n]({{\Phi}}(\y,r ))\d
r \qquad \forall s \in (0,\t_+(\y)).\end{equation} Now, from  H\"older inequality,
\begin{multline*}
\left|\int_{s_0}^s[\T_{\mathrm{max},\,p} f_n]({{\Phi}}(\y,r ))\d
r - \int_{s_0}^s[\T_{\mathrm{max},\,p} f]({{\Phi}}(\y,r ))\d
r\right| \leq |s-s_{0}|^{1/q}\\
\left(\int_{s_{0}}^{s}\left|[\T_{\mathrm{max},\,p} f_n]({{\Phi}}(\y,r ))-[\T_{\mathrm{max},\,p} f]({{\Phi}}(\y,r ))\right|^{p}\d r\right)^{1/p}\end{multline*}
and the last term goes to zero as $n \to \infty$ from \eqref{eq:conv}. Hence,  one sees that the right-hand-side of \eqref{eq:fnTfn} converges  for any $s \in (0,\t_{+}(\y))$ as $n \to \infty$. 
Therefore,  the second term on the left-hand
side also must converge as $n \to \infty.$ Thus, for any $s \in
(0,\t_+(\y))$, the limit
$$\lim_{n \to \infty}f_n({{\Phi}}(\y,s))=\widetilde{f}_-({{\Phi}}(\y,s ))$$
exists and, for any $0 < s < \t_+(\y)$
\begin{equation*}\label{d/ds}
\widetilde{f}_-({{\Phi}}(\y,s
))=\widetilde{f}_-({{\Phi}}(\y,s_0 ))-\int_{s_0}^s
[\T_{\mathrm{max},\,p} f]({{\Phi}}(\y,r ))\d r.\end{equation*}  It
is easy to check then that $\widetilde{f}_-(\x)=f(\x)$ for almost
every $\x \in \O_-$. The same arguments lead to the existence of
$\widetilde{f}_+$.
\end{proof}
The above result shows that the mild formulation of Theorem
\ref{representation} is fulfilled for any $\x \in \O_- \cup \O_+$.
It remains to deal with $\O_\infty:=\O_{-\infty}\cap \O_{+\infty}$ and one has the following whose proof is a slight modification of the one of \cite[Proposition 3.15]{mjm1} as in the above proof:
\begin{propo}\label{f+-infinity} Let $f \in \D(\T_{\mathrm{max},\,p})$. Then, there exists a set $\mathcal{O}
\subset \O_\infty$ with $\mu(\mathcal{O})=0$ and a function
$\widetilde{f}$ defined on $\{\z={{\Phi}}(\x,t),\,\x \in
\O_\infty \setminus \mathcal{O},\,t \in \mathbb{R}\}$ such that
$f(\x)=\widetilde{f}(\x)$ $\mu$-almost every $\x \in \O_\infty$ and
$$\widetilde{f}({{\Phi}}(\x,s))-\widetilde{f}({{\Phi}}(\x,t))=\int_s^t
[\T_{\mathrm{max},\,p} f ]({{\Phi}}(\x, r))\d r, \qquad \forall\,
\x \in \O_\infty \setminus \mathcal{O},\: s < t.$$
\end{propo}

Combining all the above results, the proof of Theorem \ref{representation} becomes exactly the same as that of \cite[Theorem 3.6]{mjm1}.

\section{Additional properties of $\T_{\mathrm{max},\,p}$.}

We establish here several results, reminiscent of \cite{arlo11} about how $\T_{\mathrm{max},\,p}$ and some strongly continuous family of operators can interplay. We start with the following where we recall that $\D_{0}$ has been defined in the beginning of Section \ref{sec:explicit}
\begin{propo}\label{prop:Ut}
Let $(U(t))_{t\geq0}$ be a strongly continuous family of $\mathscr{B}(X)$. For any $f \in X$, set
$$I_{t}[f]=\int_{0}^{t}U(s)f\d s, \qquad \forall t \geq 0.$$
Assume that
\begin{enumerate}[i)]
\item For any $f \in \D_{0},$ the mapping $t \in [0,\infty) \mapsto U(t)f \in X$ is differentiable with 
$$\dfrac{\d}{\d t}U(t)f=U(t)\T_{\mathrm{max},\,p}f, \qquad t \geq 0.$$
\item For any $f \in \D_{0}$ and any $t \geq 0$, it holds that $U(t)f \in \D(\T_{\mathrm{max},\,p})$ with $\T_{\mathrm{max},\,p}U(t)f=U(t)\T_{\mathrm{max},\,p}f.$
\end{enumerate}
Then, the following holds
\begin{enumerate}
\item for any $f \in X$ and $t >0$, $I_{t}[f] \in \D(\T_{\mathrm{max},\,p})$ with
$$\T_{\mathrm{max},\,p}I_{t}[f]=U(t)f-U(0)f.$$
\item for any $f \in \D_{0}$ the mapping $t \in [0,\infty) \mapsto \B^{\pm}U(t)f \in Y_{p}^{\pm}$ is continuous and, 
$$\B^{\pm}I_{t}[f]=\int_{0}^{t}\B^{\pm}U(s)f \d s \qquad \forall t >0.$$
\end{enumerate}

Let now $f \in X$ be such that  $\B^{-}I_{t}[f] \in \lm$ for all $t >0$ with $t \in [0,\infty) \mapsto \B^{-}I_{t}[f] \in \lm \text{ continuous,}$
then,
$$\B^{+}I_{t}[f] \in \lp
\qquad \text{ and } \quad t \in [0,\infty) \mapsto \B^{+}I_{t}[f] \in \lp \text{ continuous.}$$
\end{propo}
\begin{proof} Under assumptions $i)-ii)$, for any $f \in \D_{0}$, since both the mappings $t \mapsto U(t)f$ and $t \mapsto \T_{\mathrm{max},\,p}U(t)f$ are continuous and $\T_{\mathrm{max},\,p}$ is closed one has $I_{t}[f] \in \D(\T_{\mathrm{max},\,p})$ with 
$$\T_{\mathrm{max},\,p}I_{t}[f]=\int_{0}^{t}\T_{\mathrm{max},\,p}U(s)f\d s=\int_{0}^{t}\dfrac{\d}{\d s}U(s)f \d s=U(t)f-U(0)f.$$
This proves that \textit{(1)} holds for $f \in \D_{0}$ and, since $\D_{0}$ is dense in $X$, the result holds for any $f \in X$. 

Let us prove \textit{(2)}. Pick $f \in \D_{0}$. Since the mapping $t \geq 0\mapsto U(t)f \in \D(\T_{\mathrm{max},\,p})$ is continuous for the graph norm on $\D(\T_{\mathrm{max},\,p})$ while $\B^{\pm}\::\:\D(\T_{\mathrm{max},\,p}) \mapsto Y_{p}^{\pm}$ is continuous (see Remark \ref{nb:graph-norm}), the mapping $t \geq 0 \mapsto \B^{\pm}U(t)f \in Y_{p}^{\pm}$ is continuous. Moreover, $\B^{\pm}I_{t}[f]=\int_{0}^{t}\B^{\pm}U(s)f \d s$ still thanks to the continuity on $\D(\T_{\mathrm{max},\,p})$ with the graph norm and point \textit{(1)}.

Let now $f \in X$ be given such that $\B^{-}I_{t}[f] \in \lm$ for all $t >0$ with $t \geq 0 \mapsto \B^{-}I_{t}[f] \in \lm$ continuous. 
For all $t \geq 0,$ $h \geq -t$, denote now
$$I_{t,t+h}[f]=\int_{t}^{t+h}U(s)f\d s=I_{t+h}[f]-I_{t}[f].$$
Since $I_{t}[f] \in \D(\T_{\mathrm{max},\,p})$ and $\B^{-}I_{t}[f] \in \lm$, one has clearly $I_{t,t+h}[f] \in \D(\T_{\mathrm{max},\,p})$ and $\B^{-}I_{t,t+h}[f] \in \lm$ and Green's formula \eqref{greenform} yields
\begin{multline*}
\|\B^{+}I_{t,t+h}[f]\|_{\lp}^{p}=\|\B^{-}I_{t,t+h}[f]\|_{\lm}^{p}\\
-p\int_{\O}\left|I_{t,t+h}[f]\right|^{p-1}\mathrm{sign}(I_{t,t+h}[f])\T_{\mathrm{max},\,p}I_{t,t+h}[f]\d\mu\\
\leq \|\B^{-}I_{t,t+h}[f]\|_{\lm}^{p}+p\|I_{t,t+h}[f]\|_{p}^{p-1}\,\|\T_{\mathrm{max},\,p}I_{t,t+h}[f]\|_{p}.\end{multline*}
Since $\T_{\mathrm{max},\,p}I_{t,t+h}[f]=U(t+h)f-U(t)f$, one gets
\begin{multline}\label{eq:bith}
\|\B^{+}\left(I_{t+h}[f]-I_{t}[f]\right)\|_{\lp}^{p}\leq \|\B^{-}\left(I_{t+h}[f]-I_{t}[f]\right)\|_{\lm}^{p}\\
+p\|I_{t+h}[f]-I_{t}[f]\|_{p}^{p-1}\,\|U(t+h)f-U(t)f\|_{p}.\end{multline}
The continuity of $s \geq 0\mapsto U(s)f \in X$ together with the one of $s \geq 0 \mapsto \B^{-}I_{s}[f] \in \lm$ gives then that 
$$\lim_{h \to 0}\|\B^{+}\left(I_{t+h}[f]-I_{t}[f]\right)\|_{\lp}=0$$
i.e. $t \geq 0 \mapsto \B^{+}I_{t}[f] \in \lp$ is continuous.

\end{proof}

We can complement the above with the following whose proof is exactly as that of \cite[Proposition 3]{arlo11} and is omitted here:
\begin{propo} \label{propo:11-3} Let $(U(t))_{t\geq0}$ be a strongly continuous family of $\mathscr{B}(X)$ satisfying the following, for any $f \in \D_{0}$:
\begin{enumerate}[i)]
\item For any $t \geq 0$, 
$$[U(t)f](\x)=0 \qquad \forall \x \in \O \text{ such that } \t_{-}(\x) \geq t.$$
\item For any  $\y \in \Gamma_{-}$, $t >0$, $0 < r < s < \t_{+}(\y)$, it holds
$$\left[U(t)f\right](\Phi(\y,s))=\left[U(t-s+r)f\right](\Phi(\y,r)).$$
\item the mapping $t \geq 0 \mapsto U(t)f \in X$ is differentiable with $\frac{\d}{\d t}U(t)f=U(t)\T_{\mathrm{max},\,p}f$ for any $t \geq 0.$
\end{enumerate}
Then, the following properties hold
\begin{enumerate}
\item For any $f \in X$ and any $t \geq 0$ and $\mu_{-}$-a.e. $\y \in \Gamma_{-}$, given $0 < s_{1} < s_{2} < \t_{+}(\y)$, there exists $0 < r < s_{1}$ such that
$$\int_{s_{1}}^{s_{2}}\left[U(t)f\right](\Phi(\y,s))\d s=\int_{t-s_{2}+r}^{t-s_{1}+r}\left[U(\t)f\right](\Phi(\y,r))\d\t.$$
\item For any $f \in \D_{0}$ and $t \geq 0$, one has $U(t)f \in \D(\T_{\mathrm{max},\,p})$ with $\T_{\mathrm{max},\,p}U(t)f=U(t)\T_{\mathrm{max},\,p}f.$
\end{enumerate}\end{propo}

\section{On the family $(U_{k}(t))_{t \geq 0}$}

We prove that the family of operators $(U_{k}(t))_{t \geq 0}$ introduced in Definition \ref{defi:Uk} is well-defined and satisfies Theorem \ref{theo:UKT}.  The proof is made by induction and we start with a series of Lemmas (one for each of the above properties in Theorem \ref{theo:UKT}) showing that $U_{1}(t)$ enjoys all the listed properties.

As already mentioned, the fact that the mapping
$$\Phi\::\:\{(\y,s) \in \Gamma_{-}\times (0,\infty)\;;\;0 < s < \t_{+}(\y) \} \to \O_{-}$$
is a measure isomorphism, for any $f \in \D_{0}$ and $t >0$, the function $U_{1}(t)f$ is well-defined and measurable on $\O$. Moreover, using Proposition \ref{prointegra} and Fubini's Theorem:
\begin{equation*}\begin{split}
\|U_{1}(t)f\|_{p}^{p}&=\int_{\Gamma_{-}}\d\mu_{-}(\y)\int_{0}^{\t_{+}(\y)}\left|[U_{1}(t)f](\Phi(\y,s))\right|^{p}\d s\\
&=\int_{\Gamma_{-}}\d\mu_{-}(\y)\int_{0}^{\min(t,\t_{+}(\y))}\left|\left[H(\B^{+}U_{0}(t-s)f)\right](\y)\right|^{p}\d s\\
&\leq \int_{0}^{t} \|H(\B^{+}U_{0}(t-s)f)\|_{\lm}^{p}\d s. 
\end{split}
\end{equation*}
Therefore, 
$$\|U_{1}(t)f\|_{p}^{p} \leq \verti{H}^{p}\int_{0}^{t}\|\B^{+}U_{0}(t-s)f\|_{\lp}^{p}\d s=\verti{H}^{p}\left(\|f\|_{p}^{p}-\|U_{0}(t)f\|_{p}^{p}\right)$$
thanks to Proposition \ref{prop:B+U0}. Therefore $\|U_{1}(t)f\|_{p} \leq\verti{H}\,\|f\|_{p}$ for all $f \in \D_{0}$ with moreover
\begin{equation}\label{eq:limU1}
\lim_{t\to0^{+}}\|U_{1}(t)f\|_{p}=0 \qquad \forall f \in \D_{0}.\end{equation} 
Since $\D_{0}$ is dense in $X$, this allows to define a unique extension operator, still denoted by $U_{1}(t) \in \mathscr{B}(X)$ with 
$$\|U_{1}(t)\|_{\mathscr{B}(X)} \leq \verti{H}, \qquad \forall t \geq 0.$$
Now, one has the following
\begin{lemme}\label{lem:U1}
The family $(U_{1}(t))_{t\geq0}$ is strongly continuous on $X$.\end{lemme}
\begin{proof} Let $t >0$ be fixed. Set $\O_{t}=\{\x \in \O_{-}\:;\;\t_{-}(\x) \leq t \}.$ One has $[U_{1}(t)f](\x)=0$ for any $\x \in \O\setminus \O_{t}$ and any $f \in X.$ Let us fix $f \in \D_{0}$ and $h >0.$ One has
\begin{equation}\label{U1t+h}
\|U_{1}(t+h)f-U_{1}(t)f\|_{p}^{p}=\int_{\O_{t}}\left|U_{1}(t+h)f-U_{1}(t)f\right|^{p}\d\mu + \int_{\O_{t+h}\setminus\O_{t}}|U_{1}(t+h)f|^{p}\d\mu.\end{equation}
Now,
$$\int_{\O_{t}}\left|U_{1}(t+h)f-U_{1}(t)f\right|^{p}\d\mu=\int_{\Gamma_{-}}\d\mu_{-}(\y)\int_{0}^{\t_{+}(\y)}\left|\left[U_{1}(t+h)f-U_{1}(t)f\right](\Phi(\y,s))\right|^{p}\d s$$
and, repeating the reasoning before Lemma \ref{lem:U1} one gets
$$\int_{\O_{t}}\left|U_{1}(t+h)f-U_{1}(t)f\right|^{p}\d\mu \leq \verti{H}^{p}\int_{0}^{t}\|\B^{+}\left(U_{0}(s+h)f-U_{0}(s)f\right)\|_{\lp}^{p}\d s.$$
Since $U_{0}(s+h)f-U_{0}(s)f=U_{0}(s)\left(U_{0}(h)f-f\right)$ one gets from Proposition \ref{prop:B+U0} that
$$\int_{\O_{t}}\left|U_{1}(t+h)f-U_{1}(t)f\right|^{p}\d\mu \leq \verti{H}^{p}\left(\|U_{0}(h)f-f\|_{p}^{p}-\|U_{0}(t)\left(U_{0}(h)f-f\right)\|_{p}^{p}\right).$$
This proves that 
$$\lim_{h\to 0^{+}}\int_{\O_{t}}\left|U_{1}(t+h)f-U_{1}(t)f\right|^{p}\d\mu =0.$$
Let us investigate the second integral in \eqref{U1t+h}. One first notices that,  given $\x=\Phi(\y,s)$ with $\y \in \Gamma_{-}$, $0 < s < \min(t,\t_{+}(\y))$, it holds
\begin{equation}\label{eq:Uoth}\begin{split}
\left[U_{0}(t)U_{1}(h)f\right](\x)&=\chi_{\{t < \t_{-}(\x)\}}\left[U_{1}(h)f\right](\Phi(\x,-t))=\chi_{(t,\infty)}(s)\,\left[U_{1}(h)f\right](\Phi(\y,s-t))\\
&=\chi_{(t,t+h]}(s)\left[H\left(\B^{+}U_{0}(t+h-s)f\right)\right](\y)\\
&=\chi_{(t,t+h]}(s)\left[U_{1}(t+h)f\right](\Phi(\y,s))=\chi_{\{t < \t_{-}(\x)\}}\left[U_{1}(t+h)f\right](\x).\end{split}\end{equation}
Therefore
$$\int_{\O_{t+h}\setminus \O_{t}}|U_{1}(t+h)f|^{p}\d\mu=\|U_{0}(t)U_{1}(h)f\|_{p}^{p}$$
and, since $(U_{0}(t))_{t\geq0}$ is a contraction semigroup, we get
$$\int_{\O_{t+h}\setminus \O_{t}}|U_{1}(t+h)f|^{p}\d\mu \leq \|U_{1}(h)f\|_{p}^{p}.$$
Using \eqref{eq:limU1}, we get $\lim_{h\to 0^{+}}\int_{\O_{t+h}\setminus \O_{t}}|U_{1}(t+h)f|^{p}\d\mu=0$ and we obtain finally that
$\lim_{h \to 0^{+}}\|U_{1}(t+h)f-U_{1}(t)f\|_{p}^{p}=0.$
One argues in a similar way for negative $h$ and gets
$$\lim_{h \to 0}\|U_{1}(t+h)f-U_{1}(t)f\|_{p}=0, \qquad \forall f \in \D_{0}.$$
Since $\D_{0}$ is dense in $X$ and $\|U_{1}(t)\|_{\mathscr{B}(X)} \leq \verti{H}$ we deduce that above limit vanishes for all $f \in X$. This proves the result.
\end{proof}
One has also the following
\begin{lemme}
For all $t \geq 0,$ $h \geq 0$ and $f \in X$ it holds
$$U_{1}(t+h)f=U_{0}(t)U_{1}(h)f+U_{1}(t)U_{0}(h)f.$$
\end{lemme}
\begin{proof} It is clearly enough to consider $t >0$, $h >0$ since $U_{1}(0)f=0$ while $U_{0}(0)$ is the identity operator. Notice that, for any $f \in \D_{0}$ and any $0 \leq t_{1} \leq t_{2}$, for $\x=\Phi(\y,s) \in\O_{t_{1}}$ we  have
$$\int_{t_1}^{t_2}[U_1(\tau)f](\x)\d \tau=\left[H\B^{+}\int_{t_{1}-s}^{t_{2}-s}U_0(\tau)f\d \tau\right](\y).$$
Now, given $f \in X$ and $0 \leq t_{1} \leq t_{2}$ the above formula is true for almost every $\x=\Phi(\y,s) \in\O_{t_{1}}$ by a density argument. Therefore, for almost every $\x=\Phi(\y,s) \in \O_{t}$ and any $\delta >0$ it holds
$$\int_{t}^{t+\delta}[U_1(r+h)f](\x)\d r=\left[H\B^+\int_{t-s}^{t+\delta-s}U_0(r+h)f\d r\right](\y)=\left[H\B^+\int_{t-s}^{t+\delta-s}U_0(r)U_0(h)f\d r\right](\y)$$
so that, using the definition of $U_{1}(r)$ again
$$\int_{t}^{t+\delta}[U_1(r+h)f](\x)\d r=\int_{t}^{t+\delta}[U_1(r)U_0(h)f](\x)\d r \qquad \forall \delta >0$$
from which we deduce that $U_1(t+h)f(\x)=U_1(t)U_0(h)f(\x)$ for almost any $\x \in \O_{t}.$With the notations of the previous proof, one has from \eqref{eq:Uoth} that
$$U_{1}(t+h)f(\x)=\left[U_{0}(t)U_{1}(h)f\right](\x) \qquad \text{ for a.e. } \x \in \O_{t+h}\setminus \O_{t}$$
This  proves the result, since $U_{1}(t)f$ vanishes on $\O_{t+h}\setminus\O_{t}$ while $U_{0}(t)f$ vanishes on $\O_{t}.$
\end{proof}

One has now the following
\begin{lemme}
For any $f \in \D_{0}$, the mapping $t \geq 0 \mapsto U_{1}(t)f \in X$ is differentiable with 
$\frac{\d}{\d t}U_{1}(t)f=U_{1}(t)\T_{\mathrm{max},\,p}f$ for any $t \geq 0.$
\end{lemme}
\begin{proof} In virtue of the previous Lemma, it is enough to prove that $t \geq 0 \mapsto U_{1}(t)f \in X$ is differentiable at $t=0$ with 
$$\frac{\d}{\d t}U_{1}(t)f\vert_{t=0}=U_{1}(0)\T_{\mathrm{max},\,p}f=0.$$ Consider $t >0$. One has
$$\|U_{1}(t)f\|_{p}^{p}\leq \verti{H}^{p}\,\int_{0}^{t}\|\B^{+}U_{0}(s)f\|_{\lp}^{p}\d s.$$
Now, since $f \in \D(\T_{0,\,p})$, one has from \eqref{eq:b0?}
$$\int_{0}^{t}\|\B^{+}U_{0}(s)f\|_{\lp}^{p}\d s
\leq \frac{t^{p}}{p}\int_{0}^{t}\left(\int_{\Gamma_{+}}\left|[\T_{\mathrm{max},\,p}f](\Phi(\z,-s))\right|^{p}
\chi_{\{s < \t_{-}(\z)\}}\d\mu_{+}(\z)\right)\d s$$
so that
\begin{equation}\label{eq:uot/t}
\frac{\|U_{1}(t)f\|_{p}^{p}}{t^{p}} \leq \frac{ \verti{H}^{p}}{p}\int_{0}^{t}\left(\int_{\Gamma_{+}}\left|[\T_{\mathrm{max},\,p}f](\Phi(\z,-s))\right|^{p}
\chi_{\{s < \t_{-}(\z)\}}\d\mu_{+}(\z)\right)\d s.\end{equation}
Using Proposition \ref{prointegra}, one has
$$\int_{0}^{\infty}\left(\int_{\Gamma_{+}}\left|[\T_{\mathrm{max},\,p}f](\Phi(\z,-s))\right|^{p}
\chi_{\{s < \t_{-}(\z)\}}\d\mu_{+}(\z)\right)\d s=\|\T_{\mathrm{max},\,p}f\|_{p}^{p} < \infty$$
so that \eqref{eq:uot/t} yields
$$\lim_{t \to 0^{+}}\frac{\|U_{1}(t)f\|_{p}}{t}=0.$$
This proves the result.\end{proof}

\begin{lemme}
For any $f \in \D_{0}$ and any $t >0$, one has $U_{1}(t)f \in \D(\T_{\mathrm{max},\,p})$ with $\T_{\mathrm{max},\,p}U_{1}(t)f=U_{1}(t)\T_{\mathrm{max},\,p}f$.
\end{lemme}
\begin{proof} The proof follows from a simple application of Proposition \ref{propo:11-3} where the assumptions \textit{i)--iii)} are met thanks to the previous Lemmas.
\end{proof}

Let us now establish the following
\begin{lemme}\label{lem:prop6} For any $f \in X$ and any $t >0$, one has $\I^{1}_{t}[f]:=\int_{0}^{t}U_{1}(s)f \d s \in \D(\T_{\mathrm{max},\,p})$ with 
$$\T_{\mathrm{max},\,p}\I^{1}_{t}[f]=U_{1}(t)f,$$
and $\B^{\pm} \I_{t}^{1}[f] \d s \in L^{p}_{\pm}$,
\begin{equation}\label{eq:bI1t}
\B^{-}\I_{t}^{1}[f]=H\B^{+}\int_{0}^{t}U_{0}(s)f \d s.\end{equation}
Moreover the mappings $t \geq 0\mapsto \B^{\pm} \I_{t}^{1}[f] \d s \in L^{p}_{\pm}$ are continuous.
Finally, for any $f \in \D_{0}$ and any $t \geq 0$, the traces $\B^{\pm}U_{1}(t)f \in L^{p}_{\pm}$ and the mappings $t \geq 0\mapsto \B^{\pm}U_{1}(t)f \in L^{p}_{\pm}$ are continuous. 
\end{lemme}
\begin{proof} Thanks to the  previous Lemmas, the family $(U_{1}(t))_{t\geq0}$ satisfies assumptions \textit{(i)--(ii)} of Proposition \ref{prop:Ut}. One deduces then from the same Proposition (point \textit{(1)}) that, for any $f \in X$ and any $t >0$, $\I^{1}_{t}[f] \in \D(\T_{\mathrm{max},\,p})$ with $\T_{\mathrm{max},\,p}\I_{t}^{1}[f]=U_{1}(t)f-U_{1}(0)f=U_{1}(t)f.$ 

In order to show that $\B^{-}\I_{t}^{1}[f] $ can be expressed through formula \eqref{eq:bI1t} we first suppose $f \in \D_{0}$. For such an $f $  both $U_{0}(t)f $ and $U_{1}(t)f$ belong to  $\D(\T_{\mathrm{max},\,p})$ for any $t\geq 0$ with $\B^{-}U_{1}(t)f=H\B^{+}U_{0}(t)f \in \lm$.
Using this equality, the continuity of $H$ and Prop. \ref{prop:Ut} (point 2) applied both to $(U_{1}(t))_{t\geq 0}$ and  $(U_{0}(t))_{t\geq 0}$  one gets
$$\B^{-}\I^{1}_{t}[f]=\int_{0}^{t}\B^{-}U_{1}(s)f \d s=\int_{0}^{t}H\B^{+}U_{0}(s)f\d s=H\left(\int_{0}^{t}\B^{+}U_{0}(s)f\d s\right)=H\B^{+}\I_{t}^{0}[f]$$
i.e.  \eqref{eq:bI1t} for $f \in \D_{0}.$

Consider now $f \in X$ and let $(f_{n})_{n} \in \D_{0}$ be such that $\lim_{n}\|f_{n}-f\|_{p}=0.$ According to Eq. \eqref{eq:b+I2}, the sequence $(\B^{+}\I_{t}^{0}[f_{n}])_{n}$ converges in $\lp$ towards $\B^{+}\I_{t}^{0}[f]$. Since \eqref{eq:bI1t} holds true for $f_{n}$, and $H$ is continuous, then the sequence $(\B^{-}\I_{t}^{1}[f_{n}])_{n}$ converges in $\lm$ to $H\B^{+}\I_{t}^{0}[f].$ One deduces from this that $\B^{-}\I_{t}^{1}[f] \in \lm$ with \eqref{eq:bI1t}. 

Moreover the mapping $t\geq0 \mapsto \B^{-}I_{t}^{1}[f] \in \lm$ is continuous since both $H$  and  the mapping $t \geq 0\mapsto \B^{+}\I_{t}^{0}[f] \in \lp$ are continuous (see Proposition \ref{prop:B+U01}).  This property and Proposition \ref{prop:Ut} imply that $ \B^{+}\I_{t}^{1}[f] \in \lp$ and that the mapping $t \mapsto \B^{+}\I_{t}^{1}[f] \in \lp$ is continuous too. 

Finally observe that, if  $f \in \D_{0}$, then $f \in \D(\T_{\mathrm{max},\,p})$ and for any $t \geq 0$ one has 
$$\I_{t}^{1}\left[\T_{\mathrm{max},\,p}f\right] =U_{1}(t)f.$$
Thus one can state that for any $f \in \D_{0}$ and any $t \geq 0$, the traces $\B^{\pm}U_{1}(t)f \in L^{p}_{\pm}$ and the mappings $t \geq 0\mapsto \B^{\pm}U_{1}(t)f \in L^{p}_{\pm}$ are continuous.  \end{proof}

Let us now investigate Property (7):
\begin{lemme}\label{lem:C6} One has
$$\int_{0}^{t}\|\B^{+}U_{1}(s)f\|_{\lp}^{p}\d s \leq \verti{H}^{p}\,\int_{0}^{t}\|\B^{+}U_{0}(s)f\|_{\lp}^{p}\d s, \qquad \forall t \geq 0, \forall f \in \D_{0}.$$
\end{lemme}
\begin{proof} Given $f \in \D_{0}$, for any $s >0$ and $\mu_{+}$-a.e. $\z \in \Gamma_{+}$:
$$\left[\B^{+}U_{1}(s)f\right](\z)=\left[H\left(\B^{+}U_{0}(s-\t_{-}(\z))f\right)\right](\Phi(\z,-\t_{-}(\z))\chi_{(0,s)}(\t_{-}(\z))$$
Thus, 
\begin{multline*}
J:=\int_{0}^{t}\|\B^{+}U_{1}(s)f\|_{\lp}^{p}\d s\\
=\int_{0}^{t}\left(\int_{\Gamma_{+}}\left|\left[H\left(\B^{+}U_{0}(s-\t_{-}(\z))\right)f\right](\Phi(\z,-\t_{-}(\z)))\right|^{p}\chi_{(0,s)}(\t_{-}(\z))\d \mu_{+}(\z)\right)\d s.\end{multline*}
Now, using Fubini's Theorem and, for a given $\z \in \Gamma_{+}$, the change of variable $s \mapsto s-\t_{-}(\z)$, we get
$$J=\int_{\Gamma_{+}}\left(\int_{0}^{\max(0,t-\t_{-}(\z))}\left|\left[H\left(\B^{+}U_{0}(s)\right)f\right](\Phi(\z,-\t_{-}(\z)))\right|^{p}\d s\right)\d\mu_{+}(\z).$$
Using Fubini's Theorem again
\begin{equation*}\begin{split}
J &\leq \int_{0}^{t}\left(\int_{\Gamma_{+}}\left|\left[H \left(\B^{+}U_{0}(s)\right)f\right](\Phi(\z,-\t_{-}(\z)))\right|^{p}\d \mu_{+}(\z)\right)\d s\\
&\leq \int_{0}^{t}\left(\int_{\Gamma_{-}}\left|\left[H \left(\B^{+}U_{0}(s)\right)f\right](\y)\right|^{p}\d\mu_{-}(\y)\right)\d s\end{split}\end{equation*}
where we used \eqref{10.51}. Therefore, it is easy to check that
$$J \leq \left\|H \right\|_{\mathscr{B}(\lp,\lp)}^{p}\int_{0}^{t}\|\B^{+}U_{0}(s)f\|_{\lp}^{p}\d s.$$
which is the desired result.
\end{proof}
We finally have the following
\begin{lemme} Given $\lambda >0$ and $f \in X$, set $F_{1}=\int_{0}^{\infty}\exp(-\lambda t)U_{1}(t)f\d t.$ Then $F_{1} \in \D(\T_{\mathrm{max},\,p})$  with $\T_{\mathrm{max},\,p}F_{1}=\lambda\,F_{1}$ 
and  $\B^{\pm}F_{1} \in L^{p}_{\pm}$ with 
$$\B^{-}F_{1}=H\B^{+}C_{\lambda}f = HG_{\lambda}f\qquad \B^{+}F_{1}=(M_{\lambda}H)G_{\lambda}f.$$
\end{lemme}
\begin{proof}  
Let us first assume $f \in \D_{0}$. Then, for any $\y \in \Gamma_{-}$, $s \in (0,\t_{+}(\y))$:
\begin{equation*}
\begin{split}
F_{1}(\Phi(\y,s))&=\int_{s}^{\infty}\exp(-\lambda t)\left[H\B^{+}U_{0}(t-s)f\right](\y)\d t\\
&=\exp(-\lambda s)\,\int_{0}^{\infty}\exp(-\lambda t)\left[H\B^{+}U_{0}(t)f\right](\y)\d t\\
&=\exp(-\lambda s)\left[H\B^{+}\left(\int_{0}^{\infty}\exp(-\lambda t)U_{0}(t)f\d t\right)\right](\y)\end{split}\end{equation*}
i.e. $F_{1}(\Phi(\y,s))=\exp(-\lambda s)\left[H\B^{+}C_{\lambda}f\right](\y).$ This exactly means that $F_{1}=\Xi_{\lambda}HG_{\lambda}f$. By a density argument, this still holds for $f \in X$ and we get the desired result easily using the properties of $\Xi_{\lambda}$ and $G_{\lambda}.$
\end{proof}

The above lemmas prove that the conclusion of Theorem \ref{theo:UKT} is true for $k=1.$ One proves then by induction that the conclusion is true for any $k \geq 1$ exactly as above. Details are left to the reader.

\end{document}